\documentclass[11pt]{article}

\usepackage[colorlinks,linkcolor=darkblue,
  citecolor=blue,urlcolor=blue]{hyperref}

\usepackage{jcd}
\usepackage{fullpage}
\usepackage[T1]{fontenc}
\usepackage{macros-zo}

\renewcommand{\theassumption}{\Alph{assumption}} 

\usepackage[numbers]{natbib}

\definecolor{darkblue}{rgb}{0,0,.5}

\begin{document}

\begin{center}
  {\LARGE{\bf{
        Optimal rates for zero-order convex optimization: \\
        the power of two function evaluations}}} \\

  \vspace{.5cm}
  {\large John C.\ Duchi$^\dagger$ ~~~ Michael I.\ Jordan$^{\ast}$ ~~~
    Martin J.\ Wainwright$^\ast$ ~~~ Andre Wibisono$^{\ast}$}
 \vspace{.2cm}
 \texttt{jduchi@stanford.edu} ~~
  \texttt{\{jordan,wainwrig,wibisono\}@berkeley.edu} \\
  \vspace{.2cm}
  {\large $^\dagger$Stanford University 
    and $^{\ast}$University of California, Berkeley} \\
  \vspace{.1cm}
  August 2014\footnote{An extended abstract of this work was presented
    at Neural Information Processing Systems (NIPS 2012)~\cite{DuchiJoWaWi12}.
    This newer work contains results on non-smooth optimization,
    uses a different argument for lower bounds that corrects errors in the
    conference version, and provides several new lower and upper bounds.}
\end{center}

\vspace*{.2cm}

\begin{abstract}
  We consider derivative-free algorithms for stochastic and
  non-stochastic convex optimization problems that use only function values
  rather than gradients.  Focusing on non-asymptotic bounds on
  convergence rates, we show that if pairs of function values are
  available, algorithms for $d$-dimensional optimization that
  use gradient estimates based on random perturbations suffer a factor
  of at most $\sqrt{d}$ in convergence rate over traditional
  stochastic gradient methods.  We establish such results for both
  smooth and non-smooth cases, sharpening previous analyses that
  suggested a worse dimension dependence, and extend our
  results to the case of multiple ($\numobs \ge 2$) evaluations.
  We complement our
  algorithmic development with information-theoretic lower bounds on
  the minimax convergence rate of such problems, establishing the
  sharpness of our achievable results up to constant (sometimes
  logarithmic) factors.
\end{abstract}

\section{Introduction}

Derivative-free optimization schemes have a long history in optimization; for
instance, see the book by~\citet{Spall03} for an overview.  Such procedures
are desirable in settings in which explicit gradient calculations may be
computationally infeasible, expensive, or impossible.  Classical techniques in
stochastic and non-stochastic optimization, including Kiefer-Wolfowitz-type
procedures~\cite[e.g.][]{KushnerYi03}, use function difference information to
approximate gradients of the function to be minimized rather than calculating
gradients. There has recently been renewed interest in optimization problems
with only functional (zero-order) information available---rather than
first-order gradient information---in optimization, machine learning, and
statistics~\cite{FlaxmanKaMc05,AgarwalDeXi10,Nesterov11,GhadimiLa13,
  AgarwalFoHsKaRa13}.

In machine learning and statistics, this interest has centered around bandit
optimization~\cite{FlaxmanKaMc05, BartlettDaHaKaRaTe08, AgarwalDeXi10}, where
a player and adversary compete, with the player choosing points $\optvar$ in
some domain $\optdomain$ and an adversary choosing a point $\statval$, forcing
the player to suffer a loss $F(\optvar; \statval)$.  The goal is to choose an
optimal point $\optvar \in \optdomain$ based only on observations of function
values $F(\optvar; \statval)$.  Applications of such bandit problems include
online auctions and advertisement selection for search engines.  Similarly,
the field of simulation-based optimization provides many examples of problems
in which optimization is performed based only on function
values~\cite{Spall03, ConnScVi09, Nesterov11}.  Additionally,
in many problems
in statistics---including graphical model inference~\cite{WainwrightJo08} and
structured-prediction~\cite{Taskar05}---the objective is defined
variationally (as the maximum of a family of functions), so explicit
differentiation may be difficult.

Despite the long history and recent renewed interest in such
procedures, a precise understanding of their convergence behavior
remains elusive.  In this paper, we study algorithms for solving
stochastic convex optimization problems of the form
\begin{equation}
  \label{eqn:objective}
  \minimize_{\optvar \in \optdomain} f(\optvar) \defeq
  \E_\statprob[F(\optvar; \statrv)] = \int_\statsamplespace F(\optvar;
  \statval) d\statprob(\statval),
\end{equation}
where $\optdomain \subseteq \R^\dim$ is a compact convex set,
$\statprob$ is a distribution over the space $\statsamplespace$, and
for $\statprob$-almost every $\statval \in \statsamplespace$, the
function $F(\cdot; \statval)$ is closed and convex.  We focus on the
convergence rates of algorithms observing only stochastic
realizations of the function values $f(\optvar)$, though our
algorithms naturally apply in the non-stochastic case as well.

One related body of work focuses on problems where, for a given value
$\statval \in \statsamplespace$ (or sample $\statrv \sim \statprob$),
it is only possible to observe $F(\optvar; \statval)$ at a single location
$\optvar$.  \citet[Chapter 9.3]{NemirovskiYu83} develop a randomized
sampling strategy that estimates the gradient $\nabla F(\optvar;
\statval)$ via randomized evaluations of function values at points $\optvar$
on the surface of an $\ell_2$-sphere.  \mbox{\citet{FlaxmanKaMc05}}
build on this approach and establish some implications for
bandit convex optimization problems.  The convergence rates given in
these early papers are sub-optimal, as more recent
work shows~\cite{AgarwalFoHsKaRa13}.  For instance,
\mbox{\citet{AgarwalFoHsKaRa13}} provide algorithms that achieve
convergence rates after $\totaliter$ iterations of $\order(\dim^{16} /
\sqrt{\totaliter})$; however, as the authors themselves note, the
algorithms are quite complicated. \mbox{\citet{JamiesonNoRe12}}
present simpler comparison-based algorithms for solving a
subclass of such
problems, and \citet{Shamir13} gives optimal algorithms for
quadratic objectives, as well as providing some lower bounds on
optimization error when only single function values are available.

Some of the difficulties inherent in optimization using only a single
function evaluation are alleviated when the function $F(\cdot;
\statval)$ can be evaluated at \emph{two} points, as noted
independently by \citet{AgarwalDeXi10} and \citet{Nesterov11}. Such
multi-point settings prove useful for optimization problems in
which observations $\statrv$ are available, yet we only have black-box
access to objective values $F(\optvar; \statrv)$; examples of such
problems include simulation-based
optimization~\cite{Nesterov11,ConnScVi09} and variational approaches
to graphical models and classification~\cite{Taskar05,WainwrightJo08}.
The essential insight underlying multi-point schemes is as follows:
for small non-zero scalar $\smoothparam$ and a vector $\perturbrv \in
\R^\dim$, the quantity $(F(\optvar + \smoothparam \perturbrv;
\statval) - F(\optvar; \statval)) / \smoothparam$ approximates a
directional derivative of $F(\optvar; \statval)$ in the direction
$\perturbrv$ that first-order schemes may exploit.
Relative to schemes based on only a single
function evaluation at each iteration, such two-sample-based gradient
estimators exhibit faster convergence rates~\cite{AgarwalDeXi10,
  Nesterov11, GhadimiLa13}.  In the current paper, we take this line
of work further, in particular by characterizing \emph{optimal}
rates of convergence over all procedures based on multiple noisy
function evaluations.  Moreover, adopting the two-point perspective,
we present simple randomization-based
algorithms that achieve these optimal rates.

More formally, we study algorithms that receive a vector of paired
observations, $\obs(\optvar, \altoptvar) \in \R^2$, where $\optvar$
and $\altoptvar$ are points selected by the algorithm. The
$\iter^{\rm th}$ observation takes the form
\begin{equation}
  \label{eqn:observation-model}
  \obs[\iter](\optvar^\iter, \altoptvar^\iter) \defeq \left[\begin{matrix}
      F(\optvar^\iter; \statrv[\iter]) \\ F(\altoptvar^\iter; \statrv[\iter])
    \end{matrix} \right],
\end{equation}
where $\statrv[\iter]$ is an independent sample drawn from the distribution
$\statprob$. After $\totaliter$ iterations, the algorithm returns a vector
$\what{\optvar}(\totaliter) \in \optdomain$.  In this setting, we analyze
stochastic gradient and mirror-descent \mbox{procedures~\cite{Zinkevich03,
    NemirovskiYu83, BeckTe03, NemirovskiJuLaSh09}} that construct gradient
estimators using the two-point observations $\obs[\iter]$ (as well as the
natural extension to $\numobs \ge 2$ observations).  By a careful analysis of
the dimension dependence of certain random perturbation schemes, we show that
the convergence rate attained by our stochastic gradient methods is roughly a
factor of $\sqrt{\dim}$ worse than that attained by stochastic methods that
observe the full gradient $\nabla F(\optvar; \statrv)$. Under appropriate
conditions, our convergence rates are a factor of $\sqrt{\dim}$ better than
those attained in past work~\cite{AgarwalDeXi10,Nesterov11}.  For smooth
problems, \mbox{\citet{GhadimiLa13}} provide results sharper than those in the
papers~\cite{AgarwalDeXi10,Nesterov11}, but do not show optimality of their
methods nor consider non-Euclidean problems.  In addition, although we present
our results in the framework of stochastic optimization, our analysis also
applies to (multi-point) bandit online convex optimization
problems~\cite{FlaxmanKaMc05, BartlettDaHaKaRaTe08, AgarwalDeXi10}, where our
results are the sharpest provided to date.  Our algorithms apply in both
smooth and non-smooth cases as well as to non-stochastic
problems~\cite{NemirovskiYu83,Nesterov11}, where our procedures give the
fastest known convergence guarantees for the non-smooth case.  Finally, by
using information-theoretic techniques for proving lower bounds in statistical
estimation, we establish that our explicit achievable rates are sharp up to
constant factors or (in some cases) factors at most logarithmic in the
dimension.

The remainder of this paper is organized as follows: in the next
section, we present our multi-point gradient estimators and
their convergence rates, providing results in
Section~\ref{sec:smooth-case} and~\ref{sec:nonsmooth-case} for smooth
and non-smooth objectives $F$, respectively. In
Section~\ref{sec:lower-bounds}, we provide information-theoretic
minimax lower bounds on
the best possible convergence rates, uniformly over all schemes based on
function evaluations.  We devote
Sections~\ref{sec:convergence-proofs} and
Section~\ref{sec:lower-bound-proofs} to proofs of the
achievable convergence rates and the lower bounds, respectively, deferring
more technical arguments to appendices.

\paragraph{Notation}
For sequences indexed by $d$, the inequality $a_d \lesssim b_d$
indicates that there is a universal numerical constant $c$ such that
$a_d \le c \cdot b_d$. For a convex function $f : \R^d \to \R$, we let
\begin{equation*}
  \partial f(\optvar) \defeq \{g \in \R^d \mid f(\altoptvar) \ge f(\optvar)
  + \<g, \altoptvar - \optvar\>, ~ \mbox{for~all~}
  \altoptvar \in \R^d\}
\end{equation*}
denote the subgradient set of $f$ at $\optvar$. We say a function $f$
is $\lambda$-strongly convex with respect to the norm $\norm{\cdot}$
if for all $\optvar, \altoptvar \in \R^d$, we have $f(\altoptvar) \ge
f(\optvar) + \<g, \altoptvar - \optvar\> + (\lambda / 2) \norm{\optvar
  - \altoptvar}^2$ for all $g \in \partial f(\optvar)$. Given a norm
$\norm{\cdot}$, we denote its dual norm by $\dnorm{\cdot}$.  We let
$\normal(0,I_{\dim \times \dim})$ denote the standard normal
distribution on $\R^\dim$.  We denote the $\ell_2$-ball in $\R^\dim$
with radius $r$ centered at $v$ by $\B^\dim(v,r)$, and
$\S^{\dim-1}(v,r)$ denotes the $(\dim-1)$-dimensional $\ell_2$-sphere
in $\R^\dim$ with radius $r$ centered at $v$.  We also use the
shorthands $\B^\dim = \B^\dim(0,1)$ and $\S^{\dim-1} =
\S^{\dim-1}(0,1)$, and $\onevec$ for the all-ones vector.


\section{Algorithms}
\label{sec:algorithms}

We begin by providing some background on the class of stochastic
mirror descent methods for solving the problem $\min_{\optvar \in
  \optdomain} f(\theta)$.  They are based on a \emph{proximal
  function} $\prox$, meaning a differentiable and strongly convex
function defined over $\optdomain$.  The proximal function defines a
Bregman divergence $\divergence: \optdomain \times \optdomain
\rightarrow \R_+$ via
\begin{equation*}
  \divergence(\optvar, \altoptvar) \defeq \prox(\optvar) -
  \prox(\altoptvar) - \<\nabla \prox(\altoptvar), \optvar -
  \altoptvar\>.
\end{equation*}
The mirror descent (MD) method generates a sequence of iterates
$\{\optvar[\iter]\}_{\iter=1}^\infty$ contained in $\optdomain$, using
stochastic gradient information to perform the update from iterate to
iterate. The algorithm is initialized at some point $\optvar[1]
\in \optdomain$.  At iterations $\iter = 1, 2, 3, \ldots$, the MD
method receives a (subgradient) vector $\subgrad[\iter] \in \R^d$,
which it uses to compute the next iterate via
\begin{equation}
  \label{eqn:mirror-descent-update}
  \optvar[\iter + 1] = \argmin_{\optvar \in \optdomain} \left\{
  \<\subgrad[\iter], \optvar\> + \frac{1}{\stepsize[\iter]}
  \divergence(\optvar, \optvar[\iter])\right\},
\end{equation}
where $\{\stepsize[\iter]\}_{\iter=1}^\infty$ is a non-increasing
sequence of positive stepsizes.

Throughout the paper, we impose two assumptions that are standard
in analysis of mirror descent methods~\cite{NemirovskiYu83, BeckTe03,
  NemirovskiJuLaSh09}.  Letting $\optvar^*$ denote a minimizer of the
problem~\eqref{eqn:objective}, the first assumption concerns
properties of the proximal function $\prox$ and the optimizaton domain
$\optdomain$.
\begin{assumption}
  \label{assumption:prox}
  The proximal function $\prox$ is $1$-strongly convex with respect to the
  norm $\norm{\cdot}$. The domain $\optdomain$ is compact, and there
  exists $\radius < \infty$ such that $\divergence(\optvar^*, \optvar)
  \le \half \radius^2$ for $\optvar \in \optdomain$.
\end{assumption}
\noindent
Our second assumption is standard for almost all first-order stochastic
gradient methods~\cite{NemirovskiJuLaSh09, Xiao10, Nesterov11}, and it holds
whenever the functions $F(\cdot; \statval)$ are $\lipobj$-Lipschitz with
respect to the norm $\norm{\cdot}$. We use $\dnorm{\cdot}$ to denote the dual
norm to $\norm{\cdot}$, and let $\gradfunc : \optdomain
\times \statsamplespace \rightarrow \R^d$ denote a measurable subgradient
selection for the functions $F$; that is, $\gradfunc(\optvar; \statval) \in
\partial F(\optvar; \statval)$ with $\E[\gradfunc(\optvar; \statrv)] \in
\partial f(\optvar)$.
\newcounter{saveassumption}
\setcounter{saveassumption}{\value{assumption}}
\begin{assumption}
  \label{assumption:lipschitz-objective}
  There is a constant $\lipobj < \infty$ such that the (sub)gradient selection
  $\gradfunc$ satisfies $\E[\dnorm{\gradfunc(\optvar; \statrv)}^2] \le
  \lipobj^2$ for $\optvar \in \optdomain$.
\end{assumption}
\noindent
When Assumptions~\ref{assumption:prox}
and~\ref{assumption:lipschitz-objective} hold, the convergence rates of
stochastic mirror descent methods are well understood.  In detail,
suppose that the variables $\statrv[\iter] \in \statsamplespace$ are
sampled i.i.d.\ according to $\statprob$.  With the assignment
$\subgrad[\iter] = \gradfunc(\optvar[\iter];
\statrv[\iter])$, let the sequence
$\{\optvar[\iter]\}_{\iter=1}^\infty$ be generated by the mirror
descent iteration~\eqref{eqn:mirror-descent-update}. Then for a
stepsize $\stepsize[\iter] = \stepsize / \sqrt{\iter}$, the
average $\what{\optvar}(\totaliter) = \frac{1}{\totaliter} \sum_{\iter
  = 1}^\totaliter \optvar[\iter]$ satisfies
\begin{equation}
  \label{eqn:simple-mirror-descent-convergence}
  \E[f(\what{\optvar}(\totaliter))] - f(\optvar^*)
  \le \frac{1}{2 \stepsize \sqrt{\totaliter}} \radius^2
  + \frac{\stepsize}{\sqrt{\totaliter}} \lipobj^2.
\end{equation}
We refer to the
papers~\cite[Section 2.3]{BeckTe03,NemirovskiJuLaSh09} for results of
this type.

For the remainder of this section, we explore the use of function
difference information to obtain subgradient estimates that can be
used in mirror descent methods to achieve statements similar to the
convergence guarantee~\eqref{eqn:simple-mirror-descent-convergence}.
We begin by analyzing the smooth case---when the instantaneous
functions $F(\cdot; \statval)$ have Lipschitz gradients---and proceed
to the more general (non-smooth) case in the subsequent section.

\subsection{Two-point gradient estimates and convergence rates: smooth case}
\label{sec:smooth-case}

Our first step is to show how to use two function values to construct nearly
unbiased estimators of the gradient of the objective function $f$
under a smoothness condition.  Using analytic methods different from
those from past work~\cite{AgarwalDeXi10,Nesterov11}, we are able to
obtain optimal dependence with the problem dimension $d$.  In more
detail, our procedure is based on a non-increasing sequence of
positive smoothing parameters
$\{\smoothparam_\iter\}_{\iter=1}^\infty$ and a distribution
$\smoothingdist$ on $\R^d$, to be specified, satisfying
$\E_\smoothingdist[\perturbrv \perturbrv^\top] = I$.  Given a
smoothing constant $\smoothparam$, vector $\perturbval$, and
observation $\statval$, we define the directional gradient estimate at
the point $\optvar$ as
\begin{equation}
  \label{eqn:smooth-gradient-estimate}
  \gradestsmooth(\optvar; \smoothparam, \perturbval, \statval)
  \defeq \frac{F(\optvar + \smoothparam \perturbval; \statval)
    - F(\optvar; \statval)}{\smoothparam} \perturbval.
\end{equation}
Using the estimator~\eqref{eqn:smooth-gradient-estimate}, we then
perform the following two steps.  First, upon receiving the point
$\statrv[\iter] \in \statsamplespace$, we sample an independent vector
$\perturbrv[\iter]$ from $\smoothingdist$ and set
\begin{equation}
  \label{eqn:gradient-estimator}
  \subgrad[\iter] =
  \gradestsmooth(\optvar[\iter]; \smoothparam_\iter,
  \perturbrv[\iter], \statrv[\iter]) =
  \frac{ F(\optvar[\iter] + \smoothparam_\iter
    \perturbrv[\iter]; \statrv[\iter]) - F(\optvar[\iter];
    \statrv[\iter])}{\smoothparam_\iter} \perturbrv[\iter].
\end{equation}
In the second step, we apply the mirror descent
update~\eqref{eqn:mirror-descent-update} to the quantity
$\subgrad[\iter]$ to obtain the next parameter $\optvar[\iter+1]$.

Intuition for the estimator~\eqref{eqn:smooth-gradient-estimate}
follows by considering directional derivatives.  The directional
derivative $f'(\optvar, z)$ of the function $f$ at the point $\optvar$
in the direction $z$ is
\begin{align*}
  f'(\optvar, z) \defeq \lim_{\smoothparam \downarrow 0}
  \frac{1}{\smoothparam} (f(\optvar + \smoothparam z) - f(\optvar)).
\end{align*}
This limit always exists when $f$ is convex~\cite[Chapter
  VI]{HiriartUrrutyLe96ab}, and if $f$ is differentiable at $\optvar$,
then \mbox{$f'(\optvar, z) = \<\nabla f(\optvar), z\>$.}  With this
background, the estimate~\eqref{eqn:smooth-gradient-estimate} is
motivated by the following fact~\cite[equation~(32)]{Nesterov11}:
whenever $\nabla f(\optvar)$ exists, we have
\begin{equation*}
  \E[f'(\optvar, \perturbrv)\perturbrv]
  = \E[\<\nabla f(\optvar), \perturbrv\> \perturbrv]
  = \E[\perturbrv \perturbrv^\top \nabla f(\optvar)]
  = \nabla f(\optvar),
\end{equation*}
where the final equality uses our assumption that $\E[\perturbrv
  \perturbrv^\top] = I$.  Consequently, given sufficiently small
choices of $\smoothparam_\iter$, the
vector~\eqref{eqn:gradient-estimator} should be a nearly unbiased
estimator of the gradient $\nabla f(\optvar[\iter])$.

In addition to the unbiasedness condition $\E_\smoothingdist[\perturbrv
  \perturbrv^\top] = I$, we require a few additional assumptions on
$\smoothingdist$. The first ensures that the estimator $\subgrad[\iter]$ is
well-defined.
\begin{assumption}
  \label{assumption:support-condition}
  The domain of the functions $F$ and support of $\smoothingdist$ satisfy
  \begin{equation}
    \label{eqn:support-condition}
    \dom F(\cdot; \statval) \supset \optdomain + \smoothparam_1
    \supp \smoothingdist
    ~~~ \mbox{for}~ \statval \in \statsamplespace.
  \end{equation}
\end{assumption}
\noindent
If we apply smoothing with Gaussian perturbation, the
containment~\eqref{eqn:support-condition} implies $\dom F(\cdot;
\statval) = \R^d$, though we still optimize over the compact set
$\optdomain$ in the update~\eqref{eqn:mirror-descent-update}.
We remark in passing
that if the condition~\eqref{eqn:support-condition} fails,
it is possible to optimize instead over the smaller domain
$(1 - \epsilon) \optdomain$, assuming w.l.o.g.\ that
$\optdomain$ has non-empty interior, so long as $\smoothingdist$ has compact
support (cf.~\citet[Algorithm 2]{AgarwalDeXi10}).
We also
impose the following properties on the smoothing distribution:
\begin{assumption}
  \label{assumption:smoothing-dist}
  For $\perturbrv \sim \smoothingdist$, the quantity $\bigznorm \defeq
  \sqrt{\E[\norm{\perturbrv}^4 \dnorm{\perturbrv}^2]}$ is finite, and
  moreover, there is a function $\sampledim: \mathbb{N} \rightarrow
  \R_+$ such that
  \begin{align}
    \label{EqnJohn}
    \E[\dnorm{\<g, \perturbrv\> \perturbrv}^2] \le \sampledim(\dim)
    \dnorm{g}^2 \quad \mbox{for any vector $g \in \R^\dim$.}
  \end{align}
\end{assumption}
\noindent
Although the quantity $\bigznorm$ is required to be finite, its value
does not appear explicitly in our theorem statements.  On the other
hand, the dimension-dependent quantity $\sampledim(\dim)$ from
condition~\eqref{EqnJohn} appears explicitly in our convergence rates.
As an example of these two quantities, suppose that we take
$\smoothingdist$ to be the distribution of the standard normal
$\normal(0, I_{\dim \times \dim})$, and use the $\ell_2$-norm
$\norm{\cdot} = \ltwo{\cdot}$.  In this case, a straightfoward
calculation shows that $\bigznorm^2 \lesssim \dim^3$ and
$\sampledim(\dim) \lesssim \dim$.

Finally, as previously stated, the analysis of this section
requires a smoothness assumption:
\begin{assumption}
  \label{assumption:random-function-smoothness}
  There is a function $\lipgrad : \statsamplespace \rightarrow \R_+$
  such that for $\statprob$-almost every $\statval
  \in \statsamplespace$, the function $F(\cdot; \statval)$ has
  $\lipgrad(\statval)$-Lipschitz continuous gradient with respect to the
  norm $\norm{\cdot}$, and moreover the quantity $\lipgrad(\statprob)
  \defeq \sqrt{\E[ (\lipgrad(\statrv))^2]}$ is finite.
\end{assumption}
%

Essential to stochastic gradient procedures---recall
Assumption~\ref{assumption:lipschitz-objective} and the
result~\eqref{eqn:simple-mirror-descent-convergence}---is that the gradient
estimator $\subgrad[\iter]$ be nearly unbiased and have small norm.
Accordingly, the following lemma provides quantitative guarantees on the
error associated with the gradient
estimator~\eqref{eqn:smooth-gradient-estimate}.

\begin{lemma}
  \label{lemma:difference-to-gradient}
  Under Assumptions~\ref{assumption:smoothing-dist}
  and~\ref{assumption:random-function-smoothness}, the gradient
  estimate~\eqref{eqn:smooth-gradient-estimate}
  has expectation
  \begin{equation}
    \label{eqn:first-mean}
    \E[\gradestsmooth(\optvar; \smoothparam, \perturbrv, \statrv)]
    = \nabla
    f(\optvar) + \smoothparam \lipgrad(\statprob)
    \direrrvec(\optvar, \smoothparam)
  \end{equation}
  for a vector $\direrrvec = \direrrvec(\optvar, \smoothparam)$
  such that $\dnorm{\direrrvec} \le \half
  \E[\norm{\perturbrv}^2\dnorm{\perturbrv}]$.  Its expected
  squared norm has the bound
  \begin{equation}
    \label{eqn:first-variance}
    \E[\dnorm{\gradestsmooth(\optvar; \smoothparam, \perturbrv,
        \statrv)}^2] \le 2 \sampledim(\dim)
    \E\left[\dnorm{\gradfunc(\optvar; \statrv)}^2\right] + \half
    \smoothparam^2 \lipgrad(\statprob)^2 \bigznorm^2.
  \end{equation}
\end{lemma}
\noindent 
See Section~\ref{appendix:proof-difference-to-gradient} for the proof.  The
bound~\eqref{eqn:first-mean} shows that the estimator $\subgrad[\iter]$ is
unbiased for the gradient up to a correction term of order
$\smoothparam_\iter$, while the second inequality~\eqref{eqn:first-variance}
shows that the second moment is---up to an order $\smoothparam_\iter^2$
correction---within a factor $\sampledim(\dim)$ of the standard second
moment $\E[\dnorm{\gradfunc(\optvar; \statrv)}^2]$.
We note in passing that the parameter $\smoothparam$ in the lemma
can be taken arbitrarily close to 0, which only makes $\gradestsmooth$ 
a better estimate of $\gradfunc$. The intuition is straightforward: with two
points, we can obtain arbitrarily accurate estimates of the directional
derivative.

Our main result in this section is the following theorem on the
convergence rate of the mirror descent method using the gradient
estimator~\eqref{eqn:gradient-estimator}.
\begin{theorem}
  \label{theorem:main-convergence}
  Under Assumptions~\ref{assumption:prox},
  \ref{assumption:lipschitz-objective}, \ref{assumption:support-condition},
  \ref{assumption:smoothing-dist},
  and~\ref{assumption:random-function-smoothness}, consider a sequence
  $\{\optvar[\iter]\}$ generated according to the mirror
  descent update~\eqref{eqn:mirror-descent-update} using the gradient
  estimator~\eqref{eqn:gradient-estimator}, with step and perturbation
  sizes
  \begin{equation*}
    \stepsize[\iter] = \stepsize \frac{\radius}{2 \lipobj
      \sqrt{\sampledim(d)} \sqrt{\iter}} ~~~ \mbox{and} ~~~
    \smoothparam_\iter = \smoothparam \frac{\lipobj
      \sqrt{\sampledim(d)}}{ \lipgrad(\statprob) \bigznorm} \cdot
    \frac{1}{\iter} \qquad \mbox{for $\iter = 1, 2, \ldots$.}
  \end{equation*}
  Then for all $\totaliter$,
  \begin{equation}
    \label{eqn:main-convergence}
    \E\left[f(\what{\optvar}(\totaliter)) - f(\optvar^*)\right]
    \le 2 \frac{\radius \lipobj \sqrt{\sampledim(d)}}{\sqrt{\totaliter}}
    \max\left\{\stepsize, \stepsize^{-1}\right\}
    + \stepsize \smoothparam^2
    \frac{\radius \lipobj \sqrt{\sampledim(d)}}{\totaliter}
    + \smoothparam \frac{\radius \lipobj \sqrt{\sampledim(d)}
      \log(2\totaliter)}{\totaliter},
  \end{equation}
  where $\what{\optvar}(\totaliter) = \frac{1}{\totaliter}
  \sum_{\iter=1}^\totaliter \optvar[\iter]$, and the expectation is
  taken with respect to the samples $\statrv$ and $\perturbrv$.
\end{theorem}

The proof of Theorem~\ref{theorem:main-convergence} builds on
convergence proofs developed in the analysis of online and stochastic
convex optimization~\cite{Zinkevich03, NemirovskiJuLaSh09,
  AgarwalDeXi10, Nesterov11}, but requires additional technical care,
since we never truly receive unbiased gradients.  We provide the proof
in Section~\ref{appendix:proof-main-convergence}. \\

Before continuing, we make a few remarks. First, the method is reasonably
robust to the selection of the step-size multiplier $\stepsize$;
\citet{NemirovskiJuLaSh09} previously noted this robustness for
gradient-based MD methods.  As long as $\stepsize[\iter] \propto 1 /
\sqrt{\iter}$, mis-specifying the multiplier $\stepsize$ results in a
scaling at worst linear in $\max\{\stepsize, \stepsize^{-1}\}$. We may also
use multiple independent random samples $\perturbrv[\iter, i]$, $i = 1, 2,
\ldots, \numobs$, in the construction of the gradient
estimator~\eqref{eqn:gradient-estimator} to obtain more accurate estimates
of the gradient via
$\subgrad[\iter] = \frac{1}{\numobs} \sum_{i=1}^\numobs
\gradestsmooth(\optvar[\iter]; \smoothparam_\iter, \perturbrv[\iter, i],
\statrv[\iter])$. See Corollary~\ref{corollary:ltwo-multiple} to follow
for an example of this construction.
In addition, the convergence rate of the method is independent of the
Lipschitz continuity constant $\lipgrad(\statprob)$ of the instantaneous
gradients $\nabla F(\cdot; \statrv)$, because, as noted following
Lemma~\ref{lemma:difference-to-gradient}, we may take $\smoothparam$
arbitrarily close to 0.
This suggests that similar results may hold for
non-differentiable functions; indeed, as we show in the next section, a
slightly more complicated construction of the estimator $\subgrad[\iter]$
leads to analogous guarantees for general non-smooth functions.

Although we have provided bounds on the expected convergence rate, it is
possible to give high-probability convergence
guarantees~\cite[cf.][]{CesaBianchiCoGe04, NemirovskiJuLaSh09} under
additional tail conditions on $\gradfunc$---for example, under the
boundedness condition $\dnorm{\gradfunc(\optvar; \statrv)} \le
\lipobj$---though obtaining sharp dimension-dependence requires care.
Additionally, while we have presented our results as convergence guarantees
for stochastic optimization problems, an inspection of our analysis in
Section~\ref{appendix:proof-main-convergence} shows that we also obtain
(expected) regret bounds for bandit online convex optimization
problems~\cite[cf.][]{FlaxmanKaMc05, BartlettDaHaKaRaTe08, AgarwalDeXi10}.

\subsubsection{Examples and corollaries}
\label{sec:convergence-corollaries}

We now provide examples of random sampling strategies that lead to
concrete bounds for the mirror descent algorithm based on the
subgradient estimator~\eqref{eqn:gradient-estimator}.  For each
corollary, we specify the norm $\norm{\cdot}$, proximal function
$\prox$, and distribution $\smoothingdist$. We then compute the values
that the distribution $\smoothingdist$ implies in
Assumption~\ref{assumption:random-function-smoothness} and apply
Theorem~\ref{theorem:main-convergence} to obtain a convergence rate.

We begin with a corollary that characterizes the convergence rate of
our algorithm with the proximal function $\prox(\optvar) \defeq \half
\ltwo{\optvar}^2$ under a Lipschitz continuity condition:

\begin{corollary}
  \label{corollary:ltwo}
  Given an optimization domain $\optdomain \subseteq \{\optvar \in
  \R^d \mid \ltwo{\optvar} \le \radius\}$, suppose that
  $\smoothingdist$ is uniform on the surface of the $\ell_2$-ball of
  radius $\sqrt{d}$, and that $\E[\ltwo{\gradfunc(\optvar;
      \statrv)}^2] \le \lipobj^2$.  Then
  \begin{equation*}
    \E\left[f(\what{\optvar}(\totaliter)) - f(\optvar^*)\right]
    \le 2\frac{\radius \lipobj \sqrt{\dim}}{\sqrt{\totaliter}}
    \max\{\stepsize, \stepsize^{-1}\}
    + \stepsize \smoothparam^2
    \frac{\radius \lipobj \sqrt{\dim}}{\totaliter}
    + \smoothparam \frac{\radius \lipobj \sqrt{\dim} \log(2\totaliter)}{
      \totaliter}.
  \end{equation*}
\end{corollary}
\begin{proof}
  Since $\ltwo{\perturbrv} = \sqrt{\dim}$, we have $\bigznorm =
  \sqrt{\E[\ltwo{\perturbrv}^6]} = \dim^{3/2}$, and by the
  assumption that
  $\E[\perturbrv\perturbrv^\top] = I$, we see that
  \begin{equation*}
    \E[\ltwo{\<g, \perturbrv\> \perturbrv}^2] = \dim \E[\<g,
      \perturbrv\>^2] = \dim \E[g^\top \perturbrv\perturbrv^\top g],
    \qquad \mbox{valid for any $g \in \R^d$}.
  \end{equation*}
  Thus Assumption~\ref{assumption:smoothing-dist} holds with
  $\sampledim(\dim) = \dim$, and the claim follows from
  Theorem~\ref{theorem:main-convergence}.
\end{proof}

The rate Corollary~\ref{corollary:ltwo} provides is the fastest
derived to date for zero-order stochastic optimization using two
function evaluations; both \citet{AgarwalDeXi10} and
\citet{Nesterov11} achieve rates of convergence of order $\radius
\lipobj \dim / \sqrt{\totaliter}$.  In concurrent work,
\citet{GhadimiLa13} provide a result (their Corollary 3.3) that
achieves a similar rate to that above, but their primary focus is on
non-convex problems.  Moreover, we show in the sequel that this
convergence rate is actually optimal.

Using multiple function evaluations yields faster convergence rates, as we
obtain more accurate estimates of the instantaneous gradients
$\gradfunc(\optvar; \statrv)$.  The following extension of
Corollary~\ref{corollary:ltwo} illustrates this effect:
\begin{corollary}
  \label{corollary:ltwo-multiple}
  In addition to the conditions of Corollary~\ref{corollary:ltwo}, let
  $\perturbrv[\iter,i]$, $i = 1, \ldots, \numobs$ be
  sampled independently according to $\smoothingdist$, and at each iteration
  of mirror descent use the gradient estimate $\subgrad[\iter] =
  \frac{1}{\numobs} \sum_{i=1}^\numobs \gradestsmooth(\optvar[\iter];
  \smoothparam_\iter, \perturbrv[\iter,i], \statrv[\iter])$ with the step
  and perturbation sizes
  \begin{equation*}
    \stepsize[\iter] = \stepsize \frac{\radius}{2 \lipobj
      \max\{\sqrt{d / \numobs}, 1\}} \cdot \frac{1}{\sqrt{\iter}}
    ~~~ \mbox{and} ~~~
    \smoothparam_\iter = \smoothparam \frac{\lipobj}{\lipgrad(\statprob)
      d^{3/2}} \cdot \frac{1}{\iter}.
  \end{equation*}
  There exists a universal constant $C \le 5$ such that for all $\totaliter$,
  \begin{equation*}
    \E\left[f(\what{\optvar}(\totaliter)) - f(\optvar^*)\right]
    \le C \frac{\radius \lipobj \sqrt{1 + d / \numobs}}{\sqrt{\totaliter}}
    \left[
      \max\{\stepsize, \stepsize^{-1}\}
      + \stepsize \smoothparam^2 \frac{1}{\sqrt{\totaliter}}
      + \smoothparam\frac{\log(2\totaliter)}{\totaliter}\right].
  \end{equation*}
\end{corollary}
\noindent
Corollary~\ref{corollary:ltwo-multiple} shows the intuitive result that,
with a number of evaluations linear in the dimension $d$, it is possible to
attain the standard (full-information) convergence rate $\radius \lipobj /
\sqrt{\totaliter}$ (cf.~\cite{AgarwalBaRaWa12}) using only function
evaluations; we are (essentially) able to estimate the gradient $\gradfunc(\optvar; \statrv)$. We provide a proof of
Corollary~\ref{corollary:ltwo-multiple} in
Section~\ref{appendix:proof-ltwo-multiple}.

In high-dimensional scenarios, appropriate choices for the proximal
function $\prox$ yield better scaling on the norm of the
gradients~\cite{NemirovskiYu83, Gentile02, NemirovskiJuLaSh09}.  In
the setting of online learning or stochastic optimization, suppose
that one observes gradients $\gradfunc(\optvar; \statrv)$.  If the
domain $\optdomain$ is the simplex, then exponentiated gradient
algorithms~\cite{KivinenWa97, BeckTe03} using the proximal function
$\prox(\optvar) = \sum_j \optvar_j \log \optvar_j$ obtain rates of
convergence dependent on the $\ell_\infty$-norm of the gradients
$\linf{\gradfunc(\optvar; \statrv)}$.  This scaling is more palatable
than bounds that depend on Euclidean norms applied to the gradient
vectors, which may be a factor of $\sqrt{d}$ larger.
Similar results apply using proximal functions based on
$\ell_p$-norms~\cite{Ben-TalMaNe01,BeckTe03}.  In our case, if we make
the choice $p = 1 + \frac{1}{\log(2\dim)}$ and $\prox(\optvar) =
\frac{1}{2(p-1)} \norm{\optvar}_p^2$, we obtain the following
corollary, which holds under the conditions of
Theorem~\ref{theorem:main-convergence}.
\begin{corollary}
  \label{corollary:p-norms}
  Suppose that \mbox{$\E[\linf{\gradfunc(\optvar; \statrv)}^2] \le
    \lipobj^2$,} the optimization domain $\optdomain$ is contained in
  the $\ell_1$-ball $\{\optvar \in \R^d \, \mid \, \lone{\optvar} \le
  \radius\}$, and $\smoothingdist$ is uniform on the hypercube
  $\{-1, 1\}^d$.  There is a universal
  constant $C \le 2\exp(1)$ such that
  \begin{equation*}
    \E\left[f(\what{\optvar}(\totaliter)) - f(\optvar^*)\right]
    \le C \frac{\radius \lipobj \sqrt{d \log (2d)}}{\sqrt{\totaliter}}
    \max\left\{\stepsize, \stepsize^{-1}\right\}
    + C \frac{\radius \lipobj \sqrt{\dim \log (2\dim)}}{\totaliter}
    \left(\stepsize \smoothparam^2 + \smoothparam \log \totaliter\right).
  \end{equation*}
\end{corollary}
\begin{proof}
  The chosen of proximal function $\prox$ is strongly convex
  with respect to the norm $\norm{\cdot}_p$ (see~\cite[Appendix
    1]{NemirovskiYu83}).  In addition, the choice $q = 1 + \log
  (2d)$ implies $1/p + 1/q = 1$, and $\norm{v}_q \le \exp(1) \linf{v}$
  for any $v \in \R^d$.  Consequently, we have $\E[\norm{\<g,
      \perturbrv\> \perturbrv}_q^2] \le \exp(2) \E[\linf{\<g, \perturbrv\>
      \perturbrv}^2]$, which allows us to apply
  Theorem~\ref{theorem:main-convergence} with the norm $\norm{\cdot}
  = \lone{\cdot}$ and the dual norm $\dnorm{\cdot} = \linf{\cdot}$.

  We claim that Assumption~\ref{assumption:smoothing-dist} is satisfied
  with $\sampledim(d) \le d$.  Since $\perturbrv \sim \uniform(\{-1,
  1\}^d)$, we have
  \begin{equation*}
    \E\left[\linf{\<g, \perturbrv\> \perturbrv}^2\right]
    = \E\left[\<g, \perturbrv\>^2\right]
    = g^\top \E[\perturbrv \perturbrv^\top] g
    = \ltwo{g}^2 \le d \linf{g}^2
    ~~ \mbox{for~any~} g \in \R^d.
  \end{equation*}
  Finally,
  we have $\bigznorm =
  \sqrt{\E[\lone{\perturbrv}^4 \linf{\perturbrv}^2]} = d^2$, which
  is finite as needed.
  By the inclusion of $\optdomain$ in the
  $\ell_1$-ball of radius $R$ and our choice of proximal
  function, we have
  \begin{equation*}
    (p - 1) \divergence(\optvar, \altoptvar)
    \le \half \norm{\optvar}_p^2 + \half \norm{\altoptvar}_p^2
    + \norm{\optvar}_p \norm{\altoptvar}_p.
  \end{equation*}
  (For instance, see Lemma 3 in the paper~\cite{Gentile02}.)  We thus
  find that $\divergence(\optvar, \altoptvar) \le 2 \radius^2 \log
  (2\dim)$ for any $\optvar, \altoptvar \in \optdomain$, and using the
  step and perturbation size choices of
  Theorem~\ref{theorem:main-convergence} gives the result.
\end{proof}
Corollary~\ref{corollary:p-norms} attains a convergence rate that
scales with dimension as $\sqrt{\dim \log \dim}$, which is a much worse
dependence on
dimension than that of (stochastic) mirror descent using
full gradient information~\cite{NemirovskiYu83,
  NemirovskiJuLaSh09}. As in Corollaries~\ref{corollary:ltwo} and~\ref{corollary:ltwo-multiple}, which have similar additional $\sqrt{d}$ factors,
the additional dependence on $\dim$ suggests
that while $\order(1 / \epsilon^2)$ iterations are required to achieve
$\epsilon$-optimization accuracy for mirror descent methods,
the two-point method requires $\order(\dim /
\epsilon^2)$ iterations to obtain the same accuracy. In
Section~\ref{sec:lower-bounds} we show that this dependence is sharp:
apart from logarithmic factors, no algorithm can attain better
convergence rates, including the problem-dependent constants $\radius$
and $\lipobj$.

\subsection{Two-point gradient estimates and convergence rates: general case}
\label{sec:nonsmooth-case}

We now turn to the general setting in which the function
$F(\cdot;\statval)$, rather than having a Lipschitz continuous
gradient, satisfies only the milder condition of Lipschitz continuity.
The difficulty in this non-smooth case is that the simple gradient
estimator~\eqref{eqn:gradient-estimator} may have overly large
norm.  For instance, a naive calculation using only
the $\lipobj$-Lipschitz continuity of the function $f$ gives the bound
\begin{equation}
\label{EqnOppossum}
  \E\left[\ltwo{(f(\optvar + \smoothparam \perturbrv) - f(\optvar))
      \perturbrv / \smoothparam}^2\right] \le \lipobj^2
  \E\left[\ltwo{\smoothparam \ltwo{\perturbrv} \perturbrv /
      \smoothparam}^2\right] = \lipobj^2 \E[\ltwo{\perturbrv}^4].
\end{equation}
This upper bound always scales at least quadratically in the
dimension, since we have the lower bound
\mbox{$\E[\ltwo{\perturbrv}^4] \ge (\E[\ltwo{\perturbrv}^2])^2 =
  \dim^2$,} where the final equality uses the assumption $\E[\perturbrv
  \perturbrv^\top] = I_{\dim \times \dim}$.  This
quadratic dependence on dimension leads to a sub-optimal convergence
rate.  Moreover, this scaling appears to be unavoidable using a single
perturbing random vector: taking $f(\optvar) = \lipobj \ltwo{\optvar}$
and setting $\optvar = 0$ shows that the bound~\eqref{EqnOppossum} may
hold with equality.

Nevertheless, the convergence rate in
Theorem~\ref{theorem:main-convergence} shows that \emph{near}
non-smoothness is effectively the same as being smooth.  This
suggests that if we can smooth the objective $f$ slightly, we may
achieve a rate of convergence even in the non-smooth case that is
roughly the same as that in Theorem~\ref{theorem:main-convergence}.
The idea of smoothing the objective has been used to obtain faster
convergence rates in both deterministic and stochastic
optimization~\cite{Nesterov05b,DuchiBaWa12}. In the stochastic
setting, \citet{DuchiBaWa12} leverage the well-known fact that
convolution is a smoothing operation, and they consider minimization
of a sequence of smoothed functions
\begin{equation}
  \label{eqn:smoothed-function}
  f_{\smoothparam}(\optvar) \defeq \E[f(\optvar
    + \smoothparam \perturbrv)]
  = \int f(\optvar + \smoothparam z) d\smoothingdist(z),
\end{equation}
where $\perturbrv \in \R^d$ has density with respect to Lebesgue
measure.  In this case, $f_\smoothparam$ is always differentiable;
moreover, if $f$ is Lipschitz, then $\nabla f_\smoothparam$ is
Lipschitz under mild conditions.

The smoothed function~\eqref{eqn:smoothed-function} leads us to a
\emph{two-point} strategy: we use a random direction as in the smooth
case~\eqref{eqn:gradient-estimator} to estimate the gradient, but we
introduce an extra step of randomization for the point at which we
evaluate the function difference. Roughly speaking, this randomness
has the effect of making it unlikely that the perturbation vector
$\perturbrv$ is near a point of non-smoothness, which allows us to
apply results similar to those in the smooth case.

More precisely, our construction uses two non-increasing sequences of
positive parameters $\{\smoothparam_{1,\iter}\}_{\iter=1}^\infty$ and
$\{\smoothparam_{2,\iter}\}_{\iter=1}^\infty$ with
$\smoothparam_{2,\iter} \le \smoothparam_{1,\iter} / 2$, and two
smoothing distributions $\smoothingdist_1$, $\smoothingdist_2$ on
$\R^\dim$. Given smoothing constants $\smoothparam_1, \smoothparam_2$,
vectors $\perturbval_1, \perturbval_2$, and
observation $\statval$, we define the (non-smooth) directional
gradient estimate at the point $\optvar$ as
\begin{equation}
  \label{eqn:non-smooth-gradient-estimate}
  \gradestns(\optvar; \smoothparam_1, \smoothparam_2,
  \perturbval_1, \perturbval_2, \statval)
  \defeq \frac{F(\optvar + \smoothparam_1 \perturbval_1
    + \smoothparam_2 \perturbval_2; \statval)
    - F(\optvar + \smoothparam_1 \perturbval_1; \statval)}{
    \smoothparam_2} \perturbval_2.
\end{equation}
Using $\gradestns$ we may define our gradient estimator, which follows the
same intuition as our construction of the stochastic
gradient~\eqref{eqn:gradient-estimator} from the smooth
estimator~\eqref{eqn:smooth-gradient-estimate}. Now, upon
receiving the point $\statrv[\iter]$, we sample
independent vectors $\perturbrv[\iter]_1 \sim \smoothingdist_1$ and
$\perturbrv[\iter]_2 \sim \smoothingdist_2$, and set
\begin{equation}
  \label{eqn:gradient-estimator-general}
  \subgrad[\iter] =
  \gradestns(\optvar[\iter]; \smoothparam_{1,\iter},
  \smoothparam_{2,\iter}, \perturbrv[\iter]_1, \perturbrv[\iter]_2,
  \statrv[\iter]) =
  \frac{ F(\optvar[\iter] + \smoothparam_{1,\iter}
    \perturbrv[\iter]_1 + \smoothparam_{2,\iter} \perturbrv[\iter]_2;
    \statrv[\iter]) - F(\optvar[\iter] + \smoothparam_{1,\iter}
    \perturbrv[\iter]_1; \statrv[\iter])}{\smoothparam_{2,\iter}}
  \perturbrv[\iter]_2.
\end{equation}
We then proceed as in the preceding section, using this estimator in
the mirror descent method. \\

\newcounter{oldassumption}
\setcounter{oldassumption}{\value{assumption}}
\setcounter{assumption}{\value{saveassumption}}
\renewcommand{\theassumption}{\Alph{assumption}$'$}

To demonstrate the convergence of gradient-based schemes with gradient
estimator~\eqref{eqn:gradient-estimator-general}, we require a few
additional assumptions. For simplicity, in this section we focus on
results for the Euclidean norm $\ltwo{\cdot}$. We impose the following
condition on the Lipschitzian properties of $F(\cdot;\statval)$, which
is a slight strengthening of
Assumption~\ref{assumption:lipschitz-objective}.
\begin{assumption}
  \label{assumption:lipschitz-ae}
  There is a function $\lipobjfun \colon \statsamplespace \rightarrow
  \R_+$ such that for $\statprob$-a.e.\ $\statval
  \in \statsamplespace$, the function $F(\cdot; \statval)$ is
  $\lipobjfun(\statval)$-Lipschitz with respect to the $\ell_2$-norm
  $\ltwo{\cdot}$, and the quantity $\lipobj(\statprob) \defeq
  \sqrt{\E[\lipobjfun(\statrv)^2]}$ is finite.
\end{assumption}

\renewcommand{\theassumption}{\Alph{assumption}}
\setcounter{assumption}{\value{oldassumption}}

\noindent
We also impose the following assumption on the smoothing distributions
$\smoothingdist_1$ and $\smoothingdist_2$.
\begin{assumption}
  \label{assumption:smoothing-dist-general}
  The smoothing distributions are one of the following pairs: (1) both
  $\smoothingdist_1$ and $\smoothingdist_2$ are standard normal
  in $\R^\dim$ with identity covariance, (2) both
  $\smoothingdist_1$ and $\smoothingdist_2$ are uniform on the
  $\ell_2$-ball of radius $\sqrt{\dim+2}$, or (3) the distribution
  $\smoothingdist_1$ is uniform on the $\ell_2$-ball of radius
  $\sqrt{\dim+2}$ and  the distribution $\smoothingdist_2$ is
  uniform on the $\ell_2$-sphere of radius $\sqrt{\dim}$.
  Additionally, we assume the containment
  \begin{equation*}
    \dom F(\cdot; \statval) \supset \optdomain + \smoothparam_{1,1}
    \supp \smoothingdist_1 + \smoothparam_{2,1} \supp \smoothingdist_2
    ~~ \mbox{for~} \statval \in \statsamplespace.
  \end{equation*}
\end{assumption}

\noindent We then have the following analog of
Lemma~\ref{lemma:difference-to-gradient}, whose proof we
provide in Section~\ref{sec:proof-difference-to-gradient-general}:

\begin{lemma}
  \label{lemma:difference-to-gradient-general}
  Under Assumptions~\ref{assumption:lipschitz-ae}
  and~\ref{assumption:smoothing-dist-general}, the gradient
  estimator~\eqref{eqn:non-smooth-gradient-estimate} has expectation
  \begin{equation}
    \label{eqn:non-smooth-expectation}
    \E[\gradestns(\optvar; \smoothparam_1, \smoothparam_2,
      \perturbrv_1, \perturbrv_2, \statrv)]
    = \nabla f_{\smoothparam_1}(\optvar)
    + \frac{\smoothparam_2}{\smoothparam_1} \lipobj(\statprob)
    \direrrvec(\optvar, \smoothparam_1, \smoothparam_2),
  \end{equation}
  where $\direrrvec = \direrrvec(\optvar, \smoothparam_1, \smoothparam_2)$
  has bound $\ltwo{\direrrvec} \le \half
  \E[\ltwo{\perturbrv_2}^3]$.  There exists a universal constant
  $c$ such that
  \begin{equation}
    \label{eqn:non-smooth-second-moment}
    \E\left[\ltwo{
        \gradestns(\optvar; \smoothparam_1, \smoothparam_2,
        \perturbrv_1, \perturbrv_2, \statrv)}^2\right]
    \le c \,
    \lipobj(\statprob)^2 \dim \left(
    \sqrt{\frac{\smoothparam_{2}}{\smoothparam_{1}}}\,\dim + 1
    + \log \dim\right).
  \end{equation}
\end{lemma}
%

Comparing Lemma~\ref{lemma:difference-to-gradient-general} to
Lemma~\ref{lemma:difference-to-gradient}, both show that one can
obtain nearly unbiased gradient of the function $f$ using two function
evaluations, but additionally, they show that the squared norm of the gradient
estimator is \emph{at most} $\dim$ times larger than the expected norm
of the subgradients $\partial F(\optvar; \statval)$, as captured by
the quantity $\lipobj^2$ from
Assumption~\ref{assumption:lipschitz-objective}
or~\ref{assumption:lipschitz-ae}. In our approach, non-smoothness
introduces an additional logarithmic penalty in the dimension; it may
be possible to remove this factor, but we do not know how at this
time. The key is that taking the second smoothing parameter
$\smoothparam_{2}$ to be small enough means that, aside from the
dimension penalty, the gradient estimator $\subgrad[\iter]$ is
essentially unbiased for $\nabla
f_{\smoothparam_{1,\iter}}(\optvar[\iter])$ and has squared norm at most
$\lipobj^2 \dim \log \dim$. This bound on size is essential for our
main result, which we now state.
\begin{theorem}
  \label{theorem:main-convergence-general}
  Under Assumptions~\ref{assumption:prox},
  \ref{assumption:lipschitz-ae},
  and~\ref{assumption:smoothing-dist-general}, consider a sequence
  $\{\optvar[\iter]\}_{\iter=1}^\infty$ generated according to the
  mirror descent update~\eqref{eqn:mirror-descent-update} using the
  gradient estimator~\eqref{eqn:gradient-estimator-general} with
  step and perturbation sizes
  \begin{equation*}
    \stepsize[\iter] = \stepsize \frac{\radius}{\lipobj(\statprob)
      \sqrt{\dim \log(2\dim)} \sqrt{\iter}}, \qquad
    \smoothparam_{1,\iter} = \smoothparam \frac{\radius}{\iter},
      \quad \text{ and } \quad
    \smoothparam_{2,\iter} = \smoothparam
      \frac{\radius}{\dim^2 \iter^2}.
  \end{equation*}
  Then there exists a universal (numerical) constant $c$ such that
  for all $\totaliter$,
  \begin{equation}
    \label{eqn:main-convergence-general}
    \E\left[f(\what{\optvar}(\totaliter)) - f(\optvar^*)\right] \le c
    \max\{\stepsize, \stepsize^{-1}\} \frac{\radius \lipobj(\statprob)
      \sqrt{\dim
        \log (2\dim)}}{\sqrt{\totaliter}} + c \smoothparam \radius
    \lipobj(\statprob) \sqrt{\dim}\, \frac{\log(2\totaliter)}{\totaliter},
  \end{equation}
  where $\what{\optvar}(\totaliter) = \frac{1}{\totaliter}
  \sum_{\iter=1}^\totaliter \optvar[\iter]$, and the expectation is
  taken with respect to the samples $\statrv$ and $\perturbrv$.
\end{theorem}
\noindent
The proof of Theorem~\ref{theorem:main-convergence-general} roughly
follows that of Theorem~\ref{theorem:main-convergence}, except that we
prove that the sequence $\optvar[\iter]$ approximately minimizes the
sequence of smoothed functions $f_{\smoothparam_{1,\iter}}$ rather
than $f$. However, for small $\smoothparam_{1,\iter}$, these two
functions are quite close, which combined with the estimates from
Lemma~\ref{lemma:difference-to-gradient-general} gives the result.
We give the full argument in
Section~\ref{appendix:proof-main-convergence-general}.

Theorem~\ref{theorem:main-convergence-general} shows that the convergence rate
of our two-point stochastic gradient algorithm for general non-smooth
functions is (at worst) a factor of $\sqrt{\log \dim}$ worse than the rate for
smooth functions in Corollary~\ref{corollary:ltwo}.  Notably, the rate of
convergence here has substantially better dimension dependence than previously
known results~\cite{AgarwalDeXi10, Nesterov11, GhadimiLa13}.



\section{Lower bounds on zero-order optimization}
\label{sec:lower-bounds}

Thus far, we have presented two main results
(Theorems~\ref{theorem:main-convergence}
and~\ref{theorem:main-convergence-general}) that provide achievable
rates for perturbation-based gradient procedures.  It is natural to
wonder whether or not these rates are sharp.  In this section, we show
that our results are---in general---unimprovable by more than a
constant factor (a logarithmic factor in dimension in the setting of
Corollary~\ref{corollary:p-norms}).  These results show that \emph{no}
algorithm exists that can achieve a faster convergence rate than those
we have presented under the oracle
model~\eqref{eqn:observation-model}.

We begin by describing the notion of minimax error.  Let $\fnclass$ be
a collection of pairs $(F, \statprob)$, each of which defines an
objective function of the form~\eqref{eqn:objective}.  Let
$\algclass_\totaliter$ denote the collection of all algorithms that
observe a sequence of data points $(\obs[1], \dots, \obs[\totaliter])
\subset \R^2$ with $\obs[\iter] = [F(\optvar[\iter], \statrv[\iter]) ~
  F(\altoptvar[\iter], \statrv[\iter])]$ and return an estimate
$\what{\optvar}(\totaliter) \in \optdomain$.  Given an algorithm $\alg
\in \algclass_\totaliter$ and a pair $(F,\statprob) \in \fnclass$, we
define the optimality gap
\begin{equation*}
  \err_\totaliter(\alg, F, \statprob, \optdomain) \defeq
  f(\what{\optvar}(\totaliter)) - \inf_{\optvar \in \optdomain}
  f(\optvar) = \E_\statprob\big[F(\what{\optvar}(\totaliter);
    \statrv)\big] - \inf_{\optvar \in \optdomain}
  \E_\statprob\left[F(\optvar; \statrv)\right],
\end{equation*}
where $\what{\optvar}(\totaliter)$ is the output of algorithm $\alg$
on the sequence of observed function values.  The
expectation of this random variable defines the \emph{minimax error}
\begin{equation}
  \label{eqn:minimax-opt-err-def}
  \mxerr_\totaliter(\fnclass, \optdomain) \defeq \inf_{\alg
    \in \algclass_\totaliter} \sup_{(F, \statprob) \in
    \fnclass} \E[\err_\totaliter(\alg, F, \statprob,
    \optdomain)],
\end{equation}
where the expectation is taken over the observations $(\obs[1], \ldots,
\obs[\totaliter])$ and any additional randomness in $\alg$.  This quantity
measures the performance of the best algorithm in
$\algclass_\totaliter$, where performance is required to be uniformly
good over the class $\fnclass$.

We now turn to the statement of our lower bounds, which are based on
simple choices of the classes $\fnclass$.  For a given
$\ell_p$-norm $\norm{\cdot}_p$, we consider the class of linear
functionals
\begin{align*}
  \fnclass_{\lipobj,p} & \defeq \{ (F, \statprob) \, \mid \,
  F(\optvar; \statval) = \< \optvar, \statval \>
  ~~ \mbox{with} ~~ \E_\statprob[\norm{\statrv}_p^2] \leq
  \lipobj^2 \big \}.
\end{align*}
Each of these function classes satisfy
Assumption~\ref{assumption:lipschitz-ae} by construction, and
moreover, $\nabla F(\cdot; \statval)$ has Lipschitz constant 0 for all
$\statval$.  We state each of our lower bounds assuming that the
domain $\optdomain$ is equal to some $\ell_q$-ball of radius
$\radius$, that is, $\optdomain = \{\optvar \in \R^d \mid
\norm{\optvar}_q \le \radius\}$.  Our first result considers the case
$p = 2$ with domain $\optdomain$ an arbitrary $\ell_q$-ball with $q \ge
1$, so we measure gradients in the $\ell_2$-norm.
\begin{proposition}
  \label{proposition:ltwo-lower-bound}
  For the class $\fnclass_{\lipobj,2}$ and $\optdomain = \{\optvar \in \R^d
  \mid \norm{\optvar}_q \le \radius\}$, we have
  \begin{equation}
    \label{eqn:ltwo-lower-bound}
    \mxerr_\totaliter(\fnclass_{\lipobj,2}, \optdomain)
    \ge \frac{1}{12}
    \left(1 - \frac{1}{q}\right) \frac{\lipobj \radius}{
      \sqrt{\totaliter}}
    \min\left\{d^{1 - 1/q}, \totaliter^{1-1/q}\right\}.
  \end{equation}
\end{proposition}
\noindent Combining the lower bound~\eqref{eqn:ltwo-lower-bound} with
our algorithmic schemes in Section~\ref{sec:algorithms} shows that they are
optimal up to constant factors.  More specifically, for $q \ge 2$, the
$\ell_2$-ball of radius $d^{1/2 - 1/q} \radius$ contains the
$\ell_q$-ball of radius $\radius$, so
Corollary~\ref{corollary:ltwo} provides an upper bound on the minimax
rate of convergence of order $\radius \lipobj \sqrt{\dim} \dim^{1/2 -
  1/q} / \sqrt{\totaliter} = \radius \lipobj d^{1 - 1/q} /
\sqrt{\totaliter}$ in the smooth case, while for $\totaliter \ge d$,
Proposition~\ref{proposition:ltwo-lower-bound} provides the lower
bound $\radius \lipobj d^{1 - 1/q} / \sqrt{\totaliter}$.
Theorem~\ref{theorem:main-convergence-general}, providing a rate of
$\radius \lipobj \sqrt{\dim \log \dim} / \sqrt{\totaliter}$ in the
general (non-smooth) case, is also tight to within logarithmic
factors.  Consequently, the stochastic gradient descent
algorithm~\eqref{eqn:mirror-descent-update} coupled with the sampling
strategies~\eqref{eqn:gradient-estimator}
and~\eqref{eqn:gradient-estimator-general} is optimal for stochastic
problems with two-point feedback.

We can prove a parallel lower bound that
applies when using multiple ($\numobs \ge 2$) function
evaluations in each iteration, that is, in the context of
Corollary~\ref{corollary:ltwo-multiple}. In this case, an inspection
of the proof of Proposition~\ref{proposition:ltwo-lower-bound}
shows that we have the bound
\begin{equation}
  \label{eqn:multi-point-lower-bound}
  \mxerr_\totaliter(\fnclass_{\lipobj,2}, \optdomain)
  \ge \frac{1}{10} \left(1 - \frac{1}{q}\right)
  \frac{\lipobj \radius}{\sqrt{\numobs \totaliter}}
  \min\left\{d^{1 - \frac{1}{q}}, \totaliter^{1 - \frac{1}{q}}\right\}.
\end{equation}
We show this in the remarks following the proof of
Proposition~\ref{proposition:ltwo-lower-bound} in
Section~\ref{sec:proof-ltwo-lower-bound}. In particular, we see that the
minimax rate of convergence over the $\ell_2$-ball is $\radius \lipobj
\sqrt{d / \numobs} / \sqrt{\totaliter}$, which approaches the 
full information minimax rate
of convergence, $\radius \lipobj / \sqrt{\totaliter}$,
as $\numobs \to d$.

For our second lower bound, we investigate the minimax rates at which
it is possible to solve stochastic convex optimization problems in
which the objective is Lipschitz continuous in the $\ell_1$-norm, or
equivalently, in which the gradients are bounded in
$\ell_\infty$-norm.  As noted earlier, such scenarios are suitable for
high-dimensional problems~\cite[e.g.][]{NemirovskiJuLaSh09}.
\begin{proposition}
  \label{proposition:lone-lower-bound}
  For the class $\fnclass_{\lipobj,\infty}$ with $\optdomain =
  \{\optvar \in \R^d \mid \lone{\optvar} \le \radius\}$, we have
  \begin{equation*}
    \mxerr_\totaliter(\fnclass_{\lipobj,\infty}, \optdomain)
    \ge \frac{1}{24} \frac{\lipobj \radius}{\sqrt{\totaliter}}
    \, \min\left\{\frac{\sqrt{\totaliter}}{\sqrt{\log(3\totaliter)}},
    \frac{\sqrt{\dim}}{\sqrt{\log(3\dim)}}\right\}.
  \end{equation*}
\end{proposition}
\noindent
This result also demonstrates the optimality of our mirror descent
algorithms up to logarithmic factors.  Recalling
Corollary~\ref{corollary:p-norms}, the MD
algorithm~\eqref{eqn:mirror-descent-update} with $\prox(\optvar)
= \frac{1}{2(p - 1)}\norm{\optvar}_p^2$, where $p = 1 + 1 /
\log(2\dim)$, implies that $\mxerr_\totaliter(\fnclass_\lipobj,
\optdomain) \lesssim \lipobj \radius \sqrt{\dim \log(2\dim)} /
\sqrt{\totaliter}$.  On the other hand,
Proposition~\ref{proposition:lone-lower-bound} provides the lower
bound $\mxerr_\totaliter(\fnclass_\lipobj, \optdomain) \gtrsim
\lipobj \radius \sqrt{d} / \sqrt{\totaliter
  \log d}$.  These upper and lower bounds match up
to logarithmic factors in dimension.

It is worth comparing these lower bounds to the achievable rates of
convergence when full gradient information is available---that is, when one
has access to the subgradient selection $\gradfunc(\optvar; \statrv)$---and
when one has access to only a single function evaluation $F(\optvar;
\statrv)$ at each iteration.  We begin with the latter, presenting a minimax
lower bound essentially due to \citet{Shamir13} for comparison.
We denote the minimax optimization error
using a single function evaluation
in each of $\totaliter$ iterations by $\mxerrone_\totaliter$. 
For the lower bound, we impose both Lipschitz conditions on the
functions $F$ and a variance condition on the observations
$F(\optvar; \statrv)$:
for a given variance $\stddev^2$, Lipschitz constant $\lipobj$, and
$\ell_p$-norm, we consider the family of optimization problems defined
by the class of convex losses
\begin{equation*}
  \fnclass_{\stddev, \lipobj, p}
  \defeq \left\{(F, \statprob) \mid
  \E_\statprob[\norm{\partial F(\optvar; \statrv)}_p^2]
  \le \lipobj^2 
  ~~ \mbox{and} ~~
  \E[(F(\optvar; \statrv) - f(\optvar))^2] \le \stddev^2 \right\}.
\end{equation*}
In our proofs, we restrict this class to functions
of the form $F(\optvar; \statval) = c_1 \lone{\optvar - c_2 \statval}$,
where $c_1, c_2$ are constants chosen to guarantee the
above inclusions. 
By an extension of techniques of \citet[Theorem 7]{Shamir13}, we have
the following proposition.
\begin{proposition}
  \label{proposition:single-observation}
  For any $p, q \ge 1$, the class $\fnclass_{\stddev, \lipobj, p}$, and
  any $\radius > 0$ with
  $\optdomain \supset \{\optvar \in \R^d \mid \norm{\optvar}_q \le \radius\}$,
  we have
  \begin{equation*}
    \mxerrone_\totaliter(\fnclass_{\stddev,\lipobj,q}, \optdomain)
    \ge \frac{1}{4} \min\left\{\frac{d \stddev}{\sqrt{\totaliter}},
    \lipobj \radius d^{1 - \frac{1}{p} - \frac{1}{q}}\right\}.
  \end{equation*}
\end{proposition}

Proposition~\ref{proposition:single-observation} shows that the asymptotic
difficulty of optimization grows at least quadratically with the dimension
$d$. Indeed, consider the Euclidean case ($p = q = 2$), and consider
minimizing $1$-Lipschitz convex functions over the $\ell_2$-ball. Assuming
that observations have variance 1, the minimax lower bound becomes
$\frac{1}{4}\min\{d / \sqrt{\totaliter}, 1\}$, so achieving $\epsilon$
accuracy requires $\Omega(d^2 / \epsilon^2)$ iterations.  This is
substantially worse than the complexity possible when using two function
evaluations: Corollary~\ref{corollary:ltwo} implies that the minimax rate of
convergence scales as $\sqrt{d / \totaliter}$, so $d / \epsilon^2$ iterations
are necessary and sufficient to achieve $\epsilon$ accuracy. By comparing with
Corollary~\ref{corollary:ltwo-multiple} and
Proposition~\ref{proposition:ltwo-lower-bound}, we see a phase
transition: with only a single function evaluation per sample $\statrv$, the
minimax rate is $d / \sqrt{\totaliter}$, yet with $\numobs \ge 2$ function
evaluations, it is possible to obtain rates of convergence scaling as
$\sqrt{d / \numobs} / \sqrt{\totaliter}$.

For the case of \emph{linear losses}---that is, when $F(\optvar; \statval) =
\<\optvar, \statval\>$---there is a smaller gap between convergence rates
possible using single function evaluation and those attainable with multiple
evaluations.  In the Euclidean case of the preceding paragraph, the minimax
convergence rate for linear losses scales (up to logarithmic factors) as
$\sqrt{d / \totaliter}$ when only a single evaluation is available (see,
e.g.,~\citet[Theorem 5.11]{BubeckCe12}). Similarly, optimization of linear
losses over the simplex (or the $\ell_1$-ball)---the classical bandit
problem of Lai and Robbins~\cite{Robbins52,LaiRo85}---scales to within
logarithmic factors as $\sqrt{d/\totaliter}$ in the paired evaluation case
(Corollary~\ref{corollary:p-norms} and
Proposition~\ref{proposition:lone-lower-bound}) and as $\sqrt{d/\totaliter}$
in the single evaluation case as well~\cite[Theorem 5]{AudibertBu09}. In the
linear case, then, there is not (generally) a phase transition between
single and multiple evaluations.

Returning now to a comparison with the full information case, each of
Propositions~\ref{proposition:ltwo-lower-bound}
and~\ref{proposition:lone-lower-bound} includes an additional
$\sqrt{d}$ factor as compared to analogous minimax
rates~\cite{AgarwalBaRaWa12,BeckTe03,NemirovskiJuLaSh09} applicable to
the case of full gradient information.  These $\sqrt{\dim}$ factors
disappear from the achievable convergence rates in
Corollaries~\ref{corollary:ltwo} and~\ref{corollary:p-norms} when one
uses $\subgrad[\iter] = \gradfunc(\optvar; \statrv)$ in the mirror
descent updates~\eqref{eqn:mirror-descent-update}.  Consequently, our
analysis shows that in the zero-order setting---in addition to
dependence on the radius $\radius$ and second moment $\lipobj^2$---any
algorithm must suffer at least an additional $\order(\sqrt{\dim})$
penalty in convergence rate, and optimal algorithms suffer precisely
this penalty.  In models for optimization in which there is a unit
cost for each function evaluation and a unit cost for obtaining a
single dimension of the gradient, the cost of using full gradient
information and that for using only function evaluations is identical;
in cases where performing $d$ function evaluations is substantially
more expensive than computing a single gradient, however, it is
preferable to use full gradient information if possible, even when the
cost of obtaining the gradients is somewhat nontrivial.



\section{Convergence proofs}
\label{sec:convergence-proofs}

We provide the proofs of the convergence results from
Section~\ref{sec:algorithms} in this section, deferring more technical
arguments to the appendices.

\subsection{Proof of Theorem~\ref{theorem:main-convergence}}
\label{appendix:proof-main-convergence}

Before giving the proof of Theorem~\ref{theorem:main-convergence}, we
state a standard lemma on the mirror descent iterates (see, for
example, \citet[Section 2.3]{NemirovskiJuLaSh09} or
\citet[Eq.~(4.21)]{BeckTe03}).
\begin{lemma}
  \label{lemma:linear-regret}
  Let $\{\subgrad[\iter]\}_{\iter=1}^\totaliter \subset \R^d$ be a
  sequence of vectors, and let $\optvar[\iter]$ be generated by the
  mirror descent iteration~\eqref{eqn:mirror-descent-update}. If
  Assumption~\ref{assumption:prox} holds, then for any $\optvar^* \in
  \optdomain$ we have
  \begin{equation*}
 \sum_{\iter=1}^\totaliter \<\subgrad[\iter], \optvar[\iter] -
 \optvar^*\> \le \frac{1}{2 \stepsize[\totaliter]} \radius^2 +
 \sum_{\iter=1}^\totaliter \frac{\stepsize[\iter]}{2}
 \dnorm{\subgrad[\iter]}^2.
  \end{equation*}
\end{lemma}

Defining the error vector $\error[\iter] \defeq \nabla
f(\optvar[\iter]) - \subgrad[\iter]$, Lemma~\ref{lemma:linear-regret}
implies that
\begin{align}
  \sum_{\iter = 1}^\totaliter \big(f(\optvar[\iter]) -
  f(\optvar^*)\big) \le \sum_{\iter=1}^\totaliter \<\nabla
  f(\optvar[\iter]), \optvar[\iter] - \optvar^*\> & =
  \sum_{\iter=1}^\totaliter \<\subgrad[\iter], \optvar[\iter] -
  \optvar^*\> + \sum_{\iter=1}^\totaliter \<\error[\iter],
  \optvar[\iter] - \optvar^*\>.  \nonumber \\ & \le \frac{1}{2
    \stepsize[\totaliter]} \radius^2 + \sum_{\iter=1}^\totaliter
  \frac{\stepsize[\iter]}{2} \dnorm{\subgrad[\iter]}^2 +
  \sum_{\iter=1}^\totaliter \<\error[\iter], \optvar[\iter] -
  \optvar^*\>.
  \label{eqn:starting-regret-bound}
\end{align}
For each iteration $\iter = 2, 3, \ldots$, let $\mc{F}_{\iter-1}$
denote the $\sigma$-field of $\statrv[1], \ldots, \statrv[\iter-1]$
and $\perturbrv[1], \ldots, \perturbrv[\iter - 1]$.  Then
Lemma~\ref{lemma:difference-to-gradient} implies $\E[\error[\iter]
  \mid \mc{F}_{\iter-1}] = \smoothparam_\iter \lipgrad(\statprob)
\direrrvec_\iter$, where $\direrrvec_\iter \equiv
\direrrvec(\optvar[\iter], \smoothparam_\iter)$ satisfies
$\dnorm{\direrrvec_\iter} \le \half \bigznorm$. Since $\optvar[\iter]
\in \mc{F}_{\iter-1}$, we can first take an expectation conditioned on
$\mc{F}_{\iter-1}$ to obtain
\begin{align*}
  \sum_{\iter=1}^\totaliter \E[\<\error[\iter], \optvar[\iter] -
    \optvar^*\>] \le \lipgrad(\statprob) \sum_{\iter=1}^\totaliter
  \smoothparam_\iter \E[\dnorm{\direrrvec_\iter} \norm{\optvar[\iter]
      - \optvar^*}] \le \half \bigznorm \radius \lipgrad(\statprob)
  \sum_{\iter=1}^\totaliter \smoothparam_\iter,
\end{align*}
where in the last step above we have used the relation
$\norms{\optvar[\iter] - \optvar^*} \le \sqrt{2 \divergence(\optvar^*,
  \optvar)} \le \radius$.  Statement~\eqref{eqn:first-variance} of
Lemma~\ref{lemma:difference-to-gradient} coupled with the assumption
that $\E[\dnorms{\gradfunc(\optvar[\iter]; \statrv)}^2 \mid
  \mc{F}_{\iter-1}] \le \lipobj^2$ yields
\begin{align*}
  \E\left[\dnorm{\subgrad[\iter]}^2\right] =
  \E\left[\E\left[\dnorm{\subgrad[\iter]}^2 \mid \mc{F}_{\iter-1}
      \right]\right] \le 2 \sampledim(d) \lipobj^2 +
  \half\smoothparam_\iter^2 \lipgrad(\statprob)^2 \bigznorm^2.
\end{align*}
Applying the two estimates above to our initial
bound~\eqref{eqn:starting-regret-bound} yields that
$\sum_{\iter=1}^\totaliter \E\big[f(\optvar[\iter]) -
  f(\optvar^*)\big]$ is upper bounded by
\begin{align}
\frac{1}{2\stepsize[\totaliter]}\radius^2 + \sampledim(d) \lipobj^2
\sum_{\iter=1}^\totaliter \stepsize[\iter] + \frac{1}{4}
\lipgrad(\statprob)^2 \bigznorm^2 \sum_{\iter=1}^\totaliter
\smoothparam_\iter^2 \stepsize[\iter] + \half \bigznorm \radius
\lipgrad(\statprob) \sum_{\iter=1}^\totaliter \smoothparam_\iter.
  \label{eqn:intermediate-regret-bound}
\end{align}

Now we use our choices of the sample size $\stepsize[\iter]$ and
$\smoothparam_\iter$ to complete the proof. For the former, we have
$\stepsize[\iter] = \stepsize \radius / (2 \lipobj
\sqrt{\sampledim(d)} \sqrt{\iter})$.  Since $\sum_{\iter=1}^\totaliter
\iter^{-\half} < \int_0^\totaliter t^{-\half} dt = 2
\sqrt{\totaliter}$, we have
\begin{equation*}
  \frac{1}{2 \stepsize[\totaliter]} \radius^2 + \sampledim(d)
  \lipobj^2 \sum_{\iter=1}^\totaliter \stepsize[\iter] \le
  \frac{\radius \lipobj \sqrt{\sampledim(d)}}{\stepsize}
  \sqrt{\totaliter} + \stepsize \radius \lipobj \sqrt{\sampledim(d)}
  \sqrt{\totaliter} \le 2 \radius \lipobj \sqrt{\sampledim(d)}
  \sqrt{\totaliter} \max\{\stepsize, \stepsize^{-1}\}.
\end{equation*}
For the second summation in the
quantity~\eqref{eqn:intermediate-regret-bound}, we have the bound
\begin{equation*}
  \stepsize \smoothparam^2 \left(\frac{\lipobj^2
    \sampledim(d)}{\lipgrad(\statprob)^2 \bigznorm^2} \right)
  \frac{\radius \lipgrad(\statprob)^2 \bigznorm^2}{ 4\lipobj
    \sqrt{\sampledim(d)}} \sum_{\iter=1}^\totaliter
  \frac{1}{\iter^{5/2}} \le \stepsize \smoothparam^2 \radius \lipobj
  \sqrt{\sampledim(d)}
\end{equation*}
since $\sum_{\iter=1}^\totaliter \iter^{-5/2} \le 4$.  The final term
in the inequality~\eqref{eqn:intermediate-regret-bound} is similarly
bounded by
\begin{equation*}
  \smoothparam \left(\frac{\lipobj \sqrt{\sampledim(d)}}{
    \lipgrad(\statprob) \bigznorm}\right) \frac{\radius
    \lipgrad(\statprob) \bigznorm}{2} (\log \totaliter + 1) =
  \smoothparam \frac{\radius \lipobj \sqrt{\sampledim(d)}}{2} (\log
  \totaliter + 1) \le \smoothparam \radius \lipobj
  \sqrt{\sampledim(d)} \log(2\totaliter).
\end{equation*}
Combining the preceding inequalities with Jensen's inequality yields
the claim~\eqref{eqn:main-convergence}.


\subsection{Proof of Lemma~\ref{lemma:difference-to-gradient}}
\label{appendix:proof-difference-to-gradient}

Let $h$ be an arbitrary convex function with $\lipgrad_h$-Lipschitz
continuous gradient with respect to the norm $\norm{\cdot}$.  Using
the tangent plane lower bound for a convex function and the
$\lipgrad_h$-Lipschitz continuity of the gradient, for any
$\smoothparam > 0$ we have
\begin{align*}
h'(\optvar, z) = \frac{\<\nabla h(\optvar), \smoothparam
  z\>}{\smoothparam} \le \frac{h(\optvar + \smoothparam z) -
  h(\optvar)}{\smoothparam} \le \frac{\<\nabla h(\optvar),
  \smoothparam z\> + (\lipgrad_h / 2) \norm{\smoothparam
    z}^2}{\smoothparam} = h'(\optvar, z) + \frac{\lipgrad_h
  \smoothparam}{2} \norm{z}^2.
\end{align*}
Consequently, for any point $\optvar \in \relint \dom h$ and for any
$z \in \R^d$, we have
\begin{equation}
\label{eqn:difference-to-directional}
\frac{h(\optvar + \smoothparam z) - h(\optvar)}{\smoothparam} z =
h'(\optvar, z) z + \frac{\lipgrad_h \smoothparam}{2} \norm{z}^2
\direrr(\smoothparam, \optvar, z) z,
\end{equation}
where $\direrr$ is some function with range contained in $[0, 1]$.
Since $\E[\perturbrv\perturbrv^\top] = I_{d \times d}$ by assumption,
equality~\eqref{eqn:difference-to-directional} implies
\begin{align}
\E\left[\frac{h(\optvar + \smoothparam \perturbrv) -
    h(\optvar)}{\smoothparam} \perturbrv \right] & =
\E\left[h'(\optvar, \perturbrv) \perturbrv + \frac{\lipgrad_h
    \smoothparam}{2} \norm{\perturbrv}^2 \direrr(\smoothparam,
  \optvar, \perturbrv) \perturbrv\right] = \nabla h(\optvar) +
\smoothparam \lipgrad_h \direrrvec(\optvar, \smoothparam),
  \label{eqn:expected-difference-to-gradient}
\end{align}
where $\direrrvec(\optvar, \smoothparam) \in \R^d$ is an error vector
with $\dnorm{\direrrvec(\optvar, \smoothparam)} \le \half
\E[\norm{\perturbrv}^2 \dnorm{\perturbrv}]$.

We now turn to proving the statements of the lemma.  Recalling the
definition~\eqref{eqn:smooth-gradient-estimate} of the gradient
estimator, we see that for $\statprob$-almost every $\statval
\in \statsamplespace$,
expression~\eqref{eqn:expected-difference-to-gradient} implies that
\begin{equation*}
  \E[\gradestsmooth(\optvar; \smoothparam, \perturbrv, \statval)] =
  \nabla F(\optvar; \statval) + \smoothparam \lipgrad(\statval)
  \direrrvec(\optvar, \smoothparam)
\end{equation*}
for some vector $\direrrvec = \direrrvec(\optvar, \smoothparam)$ with
$2 \dnorm{\direrrvec} \le \E[\norm{\perturbrv}^2 \dnorm{\perturbrv}]$.
We have $\E[\nabla F(\optvar; \statrv)] = \nabla f(\optvar[\iter])$,
and independence implies that
\begin{equation*}
  \E[\lipgrad(\statrv) \dnorm{\direrrvec(\optvar, \smoothparam)}] \le
  \sqrt{\E[\lipgrad(\statrv)^2]}\sqrt{ \E[\dnorm{\direrrvec}^2]} \le
  \half \lipgrad(\statprob) \E[\norm{\perturbrv}^2
    \dnorm{\perturbrv}],
\end{equation*}
from which the bound~\eqref{eqn:first-mean} follows.

For the second statement~\eqref{eqn:first-variance} of the lemma,
apply equality~\eqref{eqn:difference-to-directional} to $F(\cdot;
\statrv)$, obtaining
\begin{equation*}
  \gradestsmooth(\optvar; \smoothparam, \perturbrv, \statrv)
  = \<\gradfunc(\optvar, \statrv), \perturbrv\> \perturbrv
  + \frac{\lipgrad(\statrv) \smoothparam}{2}
  \norm{\perturbrv}^2 \direrr \perturbrv
\end{equation*}
for some function $\direrr \equiv \direrr(\smoothparam, \optvar,
\perturbrv, \statrv) \in [0, 1]$. The relation $(a + b)^2 \le 2a^2 +
2b^2$ then gives
\begin{align*}
\E[ \dnorm{\gradestsmooth(\optvar; \smoothparam, \perturbrv,
    \statrv)}^2] & \le \E\left[\left(\dnorm{\<\gradfunc(\optvar,
    \statrv), \perturbrv\> \perturbrv} + \half
  \dnorm{\lipgrad(\statrv) \smoothparam \norm{\perturbrv}^2 \direrr
    \perturbrv}\right)^2\right] \\ &\le 2
\E\left[\dnorm{\<\gradfunc(\optvar, \statrv), \perturbrv\>
    \perturbrv}^2\right] + \frac{\smoothparam^2}{2} \E\left[
  \lipgrad(\statrv)^2 \norm{\perturbrv}^4 \dnorm{\perturbrv}^2
  \right].
\end{align*}
Finally, Assumption~\ref{assumption:smoothing-dist} coupled with the
independence of $\statrv$ and $\perturbrv$ gives the
bound~\eqref{eqn:first-variance}.


\subsection{Proof of Corollary~\ref{corollary:ltwo-multiple}}
\label{appendix:proof-ltwo-multiple}

We show that averaging multiple directional estimates gives a gradient
estimator whose expected squared norm is smaller by a factor of $\numobs$
than that attained using a single vector $\perturbrv$. Fixing $\statval$,
let $g = \nabla F(\optvar; \statval) + \smoothparam \lipgrad(\statval)
\direrrvec(\optvar, \smoothparam, \statval)$ denote the expectation of
$\gradestsmooth(\optvar; \smoothparam, \perturbrv, \statval)$ taken over
$\perturbrv$ uniform on $\sqrt{d} \B^d$, where $2 \ltwo{\direrrvec} \le
d^{3/2}$, by equation~\eqref{eqn:expected-difference-to-gradient}.  In this
case, for $\perturbrv[i]$ drawn
i.i.d.\ $\smoothingdist$, we obtain
\begin{align*}
\E \bigg[\ltwobigg{\frac{1}{\numobs} \sum_{i=1}^\numobs
    \gradestsmooth(\optvar; \smoothparam, \perturbrv[i], \statval)}^2
  \bigg] & = \ltwo{g}^2 + \E\bigg[\ltwobigg{\frac{1}{\numobs}
    \sum_{i=1}^\numobs \gradestsmooth(\optvar; \smoothparam,
    \perturbrv[i], \statval) - g}^2 \bigg] \\ & = \ltwo{g}^2 +
\frac{1}{\numobs} \E[\ltwo{\gradestsmooth(\optvar; \smoothparam,
    \perturbrv[1], \statval) - g}^2].
\end{align*}
Now, taking an expectation over $\statrv$, we have
\begin{align*}
  \lefteqn{\E\bigg[\ltwobigg{\frac{1}{\numobs} \sum_{i=1}^\numobs
        \gradestsmooth(\optvar; \smoothparam, \perturbrv[i],
        \statrv)}^2 \bigg] \le \E[\ltwo{\nabla F(\optvar; \statrv) +
        \smoothparam \lipgrad(\statrv) \direrrvec(\optvar,
        \smoothparam, \statrv)}^2] + \frac{1}{\numobs}
    \E[\ltwo{\gradestsmooth(\optvar; \smoothparam, \perturbrv[1],
        \statrv)}^2]} \\ & \qquad\qquad\qquad ~ \stackrel{(i)}{\le} 2
  \E[\ltwo{\nabla F(\optvar; \statrv)}^2] + \frac{1}{2} \smoothparam^2
  d^3 \E[\lipgrad(\statrv)^2] + \frac{1}{\numobs} \left(2 d
  \E[\ltwo{\nabla F(\optvar; \statrv)}^2] + \half \smoothparam^2
  \lipgrad(\statprob)^2 d^3\right) \\ & \qquad\qquad\qquad ~ = 2
  \left(1 + \frac{d}{\numobs}\right) \E[\ltwo{\nabla F(\optvar;
      \statrv)}^2] + \half \left(1 + \frac{1}{\numobs}\right)
  \smoothparam^2 \lipgrad(\statprob)^2 d^3,
\end{align*}
where inequality~(i) follows from Lemma~\ref{lemma:difference-to-gradient}
and Jensen's inequality.  By comparison of this inequality with
Lemma~\ref{lemma:difference-to-gradient}'s application in
Theorem~\ref{theorem:main-convergence} and
Corollary~\ref{corollary:ltwo}---the non-$\smoothparam$-dependent
terms scale as $(1 + d/\numobs) \E[\ltwo{\nabla F(\optvar;
    \statrv)}^2]$---the stepsizes specified in the
corollary give the desired guarantee.


\subsection{Proof of Theorem~\ref{theorem:main-convergence-general}}
\label{appendix:proof-main-convergence-general}

The proof of Theorem~\ref{theorem:main-convergence-general} is similar
to that of Theorem~\ref{theorem:main-convergence}. To simplify our
proof, we first state a lemma bounding the moments of vectors that
satisfy Assumption~\ref{assumption:smoothing-dist-general}.
\begin{lemma}
\label{lemma:vector-moments}
Let the random vector $\perturbrv$ be distributed as $\normal(0, I_{d
  \times d})$, uniformly on the $\ell_2$-ball of radius $\sqrt{d +
  2}$, or uniformly on the $\ell_2$-sphere of radius $\sqrt{d}$. For
any $k \in \N$, there is a constant $c_k$ (dependent only on $k$) such
that
\begin{equation*}
\E\left[\ltwo{\perturbrv}^k\right] \le c_k d^{\frac{k}{2}}.
\end{equation*}
In all cases we have $\E[\perturbrv \perturbrv^\top] = I_{d \times
  d}$, and $c_k \le 3$ for $k = 4$ and $c_k \le \sqrt{3}$ for $k = 3$.
\end{lemma}
See Appendix~\ref{appendix:proof-vector-moments} for the proof.  We
now turn to the proof proper. From Lemmas~E.2 and~E.3 of the
paper~\cite{DuchiBaWa12}, the function~$f_{\smoothparam}$ defined
in~\eqref{eqn:smoothed-function} satisfies $f(\optvar) \le
f_{\smoothparam}(\optvar) \le f(\optvar) + \smoothparam \lipobj
\sqrt{\dim+2}$ for $\optvar \in \optdomain$. Defining the error vector
$\error[\iter] \defeq \nabla f_{\smoothparam_{1,\iter}}
(\optvar[\iter]) - \subgrad[\iter]$ and noting that $\sqrt{\dim+2} \le
\sqrt{3\dim}$, we thus have
\begin{align}
  \sum_{\iter = 1}^\totaliter \big(f(\optvar[\iter]) -
  f(\optvar^*)\big) &\le \sum_{\iter = 1}^\totaliter
  \big(f_{\smoothparam_{1,\iter}}(\optvar[\iter]) -
  f_{\smoothparam_{1,\iter}}(\optvar^*)\big) + \sqrt{3} \lipobj
  \sqrt{\dim} \sum_{\iter=1}^\totaliter \smoothparam_{1,\iter}\notag
  \\ & \le \sum_{\iter=1}^\totaliter \<\nabla
  f_{\smoothparam_{1,\iter}}(\optvar[\iter]), \optvar[\iter] -
  \optvar^*\> + \sqrt{3} \lipobj \sqrt{\dim} \sum_{\iter=1}^\totaliter
  \smoothparam_{1,\iter}\notag \\ & = \sum_{\iter=1}^\totaliter
  \<\subgrad[\iter], \optvar[\iter] - \optvar^*\> +
  \sum_{\iter=1}^\totaliter \<\error[\iter], \optvar[\iter] -
  \optvar^*\> + \sqrt{3} \lipobj \sqrt{\dim} \sum_{\iter=1}^\totaliter
  \smoothparam_{1,\iter}, \notag
\end{align}
where we have used the convexity of $f_\smoothparam$ and the
definition of $\error[\iter]$. Applying
Lemma~\ref{lemma:linear-regret} to the summed $\<\subgrad[\iter],
\optvar[\iter] - \optvar^*\>$ terms as in the proof of
Theorem~\ref{theorem:main-convergence}, we obtain
\begin{align}
  \sum_{\iter = 1}^\totaliter \big(f(\optvar[\iter]) -
  f(\optvar^*)\big) &\le \frac{\radius^2}{2 \stepsize[\totaliter]} +
  \half \sum_{\iter=1}^\totaliter \stepsize[\iter]
  \ltwo{\subgrad[\iter]}^2 + \sum_{\iter=1}^\totaliter
  \<\error[\iter], \optvar[\iter] - \optvar^*\> + \sqrt{3} \lipobj
  \sqrt{\dim} \sum_{\iter=1}^\totaliter \smoothparam_{1,\iter}.
  \label{eqn:starting-regret-bound-general}
\end{align}
The proof from this point is similar to the proof of
Theorem~\ref{theorem:main-convergence}
(cf.\ inequality~\eqref{eqn:starting-regret-bound}).  Specifically, we
bound the squared gradient $\ltwos{\subgrad[\iter]}^2$ terms, the
error $\<\error[\iter], \optvar[\iter] - \optvar^*\>$ terms, and then
control the summed $\smoothparam_\iter$ terms.  For the remainder of
the proof, we let $\mc{F}_{\iter-1}$ denote the $\sigma$-field
generated by the random variables $\statrv[1], \ldots,
\statrv[\iter-1]$, $\perturbrv[1]_1, \ldots, \perturbrv[\iter - 1]_1$,
and $\perturbrv[1]_2, \ldots, \perturbrv[\iter - 1]_2$.


\paragraph{Bounding $\<\error[\iter], \optvar[\iter] - \optvar^*\>$:}

Our first step is note that
Lemma~\ref{lemma:difference-to-gradient-general} implies
$\E[\error[\iter] \mid \mc{F}_{\iter-1}] =
\frac{\smoothparam_{2,\iter}}{\smoothparam_{1,\iter}} \lipobj
\direrrvec_\iter$, where the vector $\direrrvec_\iter \equiv
\direrrvec(\optvar[\iter], \smoothparam_{1,\iter},
\smoothparam_{2,\iter})$ satisfies $\ltwo{\direrrvec_\iter} \le \half
\E[\ltwo{\perturbrv_2}^3]$. As in the proof of
Theorem~\ref{theorem:main-convergence}, this gives
\begin{align*}
  \sum_{\iter=1}^\totaliter \E[\<\error[\iter], \optvar[\iter] - \optvar^*\>]
  \le \lipobj \sum_{\iter=1}^\totaliter
  \frac{\smoothparam_{2,\iter}}{\smoothparam_{1,\iter}}\,
  \E[\ltwo{\direrrvec_\iter} \ltwo{\optvar[\iter] - \optvar^*}]
  \le \half \, \E[\ltwo{\perturbrv_2}^3] \, \radius \lipobj
  \sum_{\iter=1}^\totaliter
  \frac{\smoothparam_{2,\iter}}{\smoothparam_{1,\iter}}.
\end{align*}
When Assumption~\ref{assumption:smoothing-dist-general} holds,
Lemma~\ref{lemma:vector-moments} implies the expectation bound
$\E[\ltwo{\perturbrv_2}^3] \le \sqrt{3} \dim^{3/2}$.  Thus
\begin{equation*}
  \sum_{\iter = 1}^\totaliter \E[\<\error[\iter],
    \optvar[\iter] - \optvar^*\>]
  \le \frac{\sqrt{3} \dim \sqrt{\dim}}{2} \radius \lipobj
  \sum_{\iter = 1}^\totaliter \frac{\smoothparam_{2,\iter}}{
    \smoothparam_{1,\iter}}.
\end{equation*}

\paragraph{Bounding $\ltwos{\subgrad[\iter]}^2$:}
Turning to the squared gradient terms from the
bound~\eqref{eqn:starting-regret-bound-general},
Lemma~\ref{lemma:difference-to-gradient-general} gives
\begin{align*}
  \E[\ltwo{\subgrad[\iter]}^2]
  = \E[\E[\ltwo{\subgrad[\iter]}^2 \mid \mc{F}_{\iter-1}]]
  \le c \, \lipobj^2 \dim \left(
    \sqrt{\frac{\smoothparam_{2,\iter}}{\smoothparam_{1,\iter}}}\,\dim
    + 1 + \log \dim\right)
  \le c' \, \lipobj^2 \dim \left(
    \sqrt{\frac{\smoothparam_{2,\iter}}{\smoothparam_{1,\iter}}}\,\dim
    + \log(2\dim)\right),
\end{align*}
where $c, c' > 0$ are numerical constants independent of
$\{\smoothparam_{1,\iter}\}$, $\{\smoothparam_{2,\iter}\}$.

\paragraph{Summing out the smoothing penalties:}
Applying the preceding estimates to our earlier
bound~\eqref{eqn:starting-regret-bound-general}, we get that
for a numerical constant $c$,
\begin{equation}
  \begin{split}
    \label{eqn:intermediate-regret-bound-general}
    \sum_{\iter=1}^\totaliter \E\big[f(\optvar[\iter])
      - f(\optvar^*)\big]
    & \le \frac{\radius^2}{2 \stepsize[\totaliter]}
    + c \lipobj^2\dim \log(2\dim)
    \sum_{\iter=1}^\totaliter \stepsize[\iter] \\
    & \qquad ~ + c \lipobj^2\dim^2
    \sum_{\iter=1}^\totaliter
    \sqrt{\frac{\smoothparam_{2,\iter}}{\smoothparam_{1,\iter}}}
    \,\stepsize[\iter]
    + \frac{\sqrt{3}}{2} \radius\lipobj\dim\sqrt{\dim}
    \sum_{\iter=1}^\totaliter
    \frac{\smoothparam_{2,\iter}}{\smoothparam_{1,\iter}}
    + \sqrt{3} \lipobj \sqrt{\dim}
    \sum_{\iter=1}^\totaliter \smoothparam_{1,\iter}.
  \end{split}  
\end{equation}
We bound the right hand side above using our choices of $\stepsize[\iter]$,
$\smoothparam_{1,\iter}$, and $\smoothparam_{2,\iter}$. We also
use the relations $\sum_{\iter=1}^\totaliter \iter^{-\half} \le
2 \sqrt{\totaliter}$ and $\sum_{\iter=1}^\totaliter \iter^{-1} \le
1 + \log \totaliter \le 2 \log \totaliter$ for $\totaliter \ge 3$.
With the setting $\stepsize[\iter] = \stepsize \radius/(\lipobj
\sqrt{\dim \log(2\dim)} \sqrt{\iter})$, the first two terms
in~\eqref{eqn:intermediate-regret-bound-general} become
\begin{align*}
  \frac{\radius^2}{2 \stepsize[\totaliter]}
  + c\lipobj^2\dim\log(2\dim)
  \sum_{\iter=1}^\totaliter \stepsize[\iter]
  & \le \frac{\radius\lipobj\sqrt{\dim \log(2\dim)}}{2\stepsize}
  \sqrt{\totaliter}
  + 2 c \stepsize \radius \lipobj \sqrt{\dim \log(2\dim)}
  \sqrt{\totaliter} \\
  & \le c' \max\{\stepsize, \stepsize^{-1}\}
  \radius \lipobj \sqrt{\dim \log(2\dim)} \sqrt{\totaliter}
\end{align*}
for a universal constant $c'$. Since we have chosen
$\smoothparam_{2,\iter}/\smoothparam_{1,\iter} = 1/(\dim^2\iter)$, we
may bound the third
term in expression~\eqref{eqn:intermediate-regret-bound-general} by
\begin{align*}
  c \lipobj^2\dim^2 \sum_{\iter=1}^\totaliter
  \sqrt{\frac{\smoothparam_{2,\iter}}{\smoothparam_{1,\iter}}}\,
  \stepsize[\iter]
  = c \lipobj^2\dim^2
  \left(\frac{\stepsize \radius}{\lipobj \sqrt{\dim \log(2\dim)}}\right)
  \frac{1}{\dim} \sum_{\iter=1}^\totaliter \frac{1}{\iter}
  \le \frac{c' \stepsize \radius \lipobj \sqrt{\dim}}{\sqrt{\log(2\dim)}}
  \log(2 \totaliter)
\end{align*}
for another universal constant $c'$.
Similarly, the fourth term in the
bound~\eqref{eqn:intermediate-regret-bound-general} becomes
\begin{align*}
  \frac{\sqrt{3}}{2}\radius\lipobj\dim\sqrt{\dim}
  \sum_{\iter=1}^\totaliter \frac{\smoothparam_{2,\iter}}{\smoothparam_{1,\iter}}
  = \frac{\sqrt{3}}{2}\radius\lipobj\dim\sqrt{\dim}\, \frac{1}{\dim^2}
  \sum_{\iter=1}^\totaliter \frac{1}{\iter}
  \le \frac{\sqrt{3} \radius \lipobj}{\sqrt{\dim}} \log(2\totaliter).
\end{align*}
Finally, since $\smoothparam_{1,\iter} = \smoothparam \radius / \iter$,
we may bound the last term in
expression~\eqref{eqn:intermediate-regret-bound-general} with
\begin{align*}
  \sqrt{3} \lipobj \sqrt{\dim}
  \sum_{\iter=1}^\totaliter \smoothparam_{1,\iter}
  = \sqrt{3} \lipobj \sqrt{\dim} \, \smoothparam \radius
  \sum_{\iter=1}^\totaliter \frac{1}{\iter}
  \le 2\sqrt{3} \smoothparam \radius \lipobj \sqrt{\dim}
  \log(2\totaliter).
\end{align*}
Using Jensen's inequality
to note that
$\E[f(\what{\optvar}(\totaliter)) - f(\optvar^*)]
\le \frac{1}{\totaliter} \sum_{\iter=1}^\totaliter
\E\left[f(\optvar[\iter]) - f(\optvar^*)\right]$
and eliminating lower-order terms,
we obtain the claim~\eqref{eqn:main-convergence-general}.


\subsection{Proof of Lemma~\ref{lemma:difference-to-gradient-general}}
\label{sec:proof-difference-to-gradient-general}


The proof of Lemma~\ref{lemma:difference-to-gradient-general} relies
on the following key technical result:
\begin{lemma}
  \label{lemma:moment-bound}
  Let $k \ge 1$ and $\smoothparam \ge 0$.
  Let $\perturbrv_1 \sim \smoothingdist_1$ and $\perturbrv_2 \sim
  \smoothingdist_2$ be independent random variables in
  $\R^\dim$, where $\smoothingdist_1$ and $\smoothingdist_2$ satisfy
  Assumption~\ref{assumption:smoothing-dist-general}. There exists a
  constant $c_k$, depending only on $k$, such that for every $1$-Lipschitz
  convex function $h$,
  \begin{equation*}
    \E\left[\left|h(\perturbrv_1 + \smoothparam \perturbrv_2) -
      h(\perturbrv_1)\right|^k\right]
      \le c_k \smoothparam^k \left[
       \smoothparam \dim^{\frac{k}{2}}
       + 1 + \log^{\frac{k}{2}}(\dim+2k) \right].
  \end{equation*}
\end{lemma}
\noindent
The proof is fairly technical, so we defer it to
Appendix~\ref{appendix:lipschitz-moment-bounds}.  It is based on the
dimension-free concentration of Lipschitz functions of standard
Gaussian vectors and vectors uniform on $\B^\dim$. \\

We return now to the proof of Lemma~\ref{lemma:difference-to-gradient-general}
proper, providing arguments for
inequalities~\eqref{eqn:non-smooth-expectation}
and~\eqref{eqn:non-smooth-second-moment}. For convenience
we recall the definition
$\lipobjfun(\statval)$ as the Lipschitz constant of $F(\cdot; \statval)$
(Assumption~\ref{assumption:lipschitz-ae}) and the
definition~\eqref{eqn:non-smooth-gradient-estimate} of the non-smooth
directional gradient
\begin{equation*}
  \gradestns(\optvar; \smoothparam_1, \smoothparam_2,
  \perturbval_1, \perturbval_2, \statval)
  = \frac{F(\optvar + \smoothparam_1 \perturbval_1
    + \smoothparam_2 \perturbval_2; \statval)
    - F(\optvar + \smoothparam_1 \perturbval_1; \statval)}{
    \smoothparam_2} \perturbval_2.
\end{equation*}
We begin with the second statement~\eqref{eqn:non-smooth-second-moment} of
Lemma~\ref{lemma:difference-to-gradient-general}. By applying
Lemma~\ref{lemma:moment-bound} to the $1$-Lipschitz convex function
$h(\altoptvar) = \frac{1}{\smoothparam_{1} \lipobjfun(\statrv)} F(\optvar +
\smoothparam_{1} \altoptvar; \statrv)$ and setting $\smoothparam =
\smoothparam_2 / \smoothparam_1$, we obtain
\begin{align}
  \E\left[\ltwo{
      \gradestns(\optvar; \smoothparam_1, \smoothparam_2,
      \perturbrv_1, \perturbrv_2, \statval)}^2\right]
  &= \frac{\smoothparam_{1}^2 \lipobjfun(\statval)^2}{
    \smoothparam_{2}^2}
  \E\left[ \left(h(\perturbrv_1 + (\smoothparam_{2} /
    \smoothparam_{1}) \perturbrv_2)
    - h(\perturbrv_1)\right)^2
    \ltwo{\perturbrv_2}^2
    \right] \nonumber \\
  & \le \frac{\lipobjfun(\statval)^2}{\smoothparam^2}
  \E\left[\left(h(\perturbrv_1 + \smoothparam
    \perturbrv_2) - h(\perturbrv_1)\right)^4\right]^\half
  \E\left[\ltwo{\perturbrv_2}^4\right]^\half.
  \label{eqn:intermediate-g-bound-general}
\end{align}
Lemma~\ref{lemma:vector-moments} implies that
$\E[\ltwo{\perturbrv_2}^4]^\half \le \sqrt{3} \dim$ for smoothing
distributions satisfying Assumption~\ref{assumption:smoothing-dist-general}.

It thus remains to bound
the first expectation in the product~\eqref{eqn:intermediate-g-bound-general}.
By Lemma~\ref{lemma:moment-bound},
\begin{equation*}
  \E\left[\left(h(\perturbrv_1 + \smoothparam
    \perturbrv_2) - h(\perturbrv_1)\right)^4 \right]
  \le c \smoothparam^4 \left[\smoothparam \dim^2
    + 1 + \log^2\dim \right]
\end{equation*}
for a numerical constant $c > 0$. Taking the square root of both sides of the
preceding display, then applying
inequality~\eqref{eqn:intermediate-g-bound-general}, yields
\begin{equation*}
  \E\left[\ltwo{
      \gradestns(\optvar; \smoothparam_1, \smoothparam_2,
      \perturbrv_1, \perturbrv_2, \statval)}^2\right]
  \le c \, \frac{\lipobjfun(\statval)^2}{
    \smoothparam^2} \,
  \smoothparam^2 \, \dim
  \left[\sqrt{\smoothparam} \dim
    + 1 + \log\dim \right].
\end{equation*}
Integrating over $\statval$ using the Lipschitz
Assumption~\ref{assumption:lipschitz-ae} proves the
inequality~\eqref{eqn:non-smooth-second-moment} in
Lemma~\ref{lemma:difference-to-gradient-general}.

For the first statement of the lemma, we define the shorthand
$F_{\smoothparam}(\optvar;\statval) = \E[F(\optvar + \smoothparam
  \perturbrv_1; \statval)]$,
where the expectation is over $\perturbrv_1 \sim \smoothingdist_1$, and note
that by Fubini's theorem, $\E[F_{\smoothparam} (\optvar;\statrv)] =
f_{\smoothparam}(\optvar)$.  By taking the expectation of
$\gradestns$ with respect to $\perturbrv_1$ only, we get
\begin{equation*}
  \E\left[
    \gradestns(\optvar; \smoothparam_1, \smoothparam_2,
    \perturbrv_1, \perturbval_2, \statval)\right]
  = \frac{F_{\smoothparam_1}(\optvar + \smoothparam_2 \perturbval_2;
    \statval) - F_{\smoothparam_1}(\optvar; \statval)}{
    \smoothparam_2} \perturbval_2.
\end{equation*}
Since $\optvar \mapsto F(\optvar; \statval)$ is
$\lipobjfun(\statval)$-Lipschitz, Lemmas~E.2(iii) and E.3(iii) of the paper
by~\citet{DuchiBaWa12} imply $F_{\smoothparam}(\cdot;\statval)$ is
$\lipobjfun(\statval)$-Lipschitz, has
$\lipobjfun(\statval)/\smoothparam$-Lipschitz continuous gradient, and
satisfies the unbiasedness condition $\E[\nabla
  F_{\smoothparam}(\optvar;\statrv)] = \nabla
f_{\smoothparam}(\optvar)$. Therefore, the same argument bounding the
bias~\eqref{eqn:first-mean} in the proof of
Lemma~\ref{lemma:difference-to-gradient} (recall
inequalities~\eqref{eqn:difference-to-directional}
and~\eqref{eqn:expected-difference-to-gradient}) yields the
claim~\eqref{eqn:non-smooth-expectation}.



\section{Proofs of lower bounds}
\label{sec:lower-bound-proofs}

We now present the proofs for our lower bounds on the minimax
error~\eqref{eqn:minimax-opt-err-def}. Our lower bounds are based on
several techniques from the statistics and information-theory
literature~\cite[e.g.][]{YangBa99,Yu97,Arias-CastroCaDa13}.  Our basic
strategy is to reduce the optimization problem to several binary
hypothesis testing problems: we choose a finite set of functions, show
that optimizing well implies that one can solve each of the binary
hypothesis tests, and then, as in statistical minimax
theory~\cite{YangBa99,Yu97}, apply divergence-based lower bounds for
the probability of error in hypothesis testing problems.


\subsection{Proof of Proposition~\ref{proposition:ltwo-lower-bound}}
\label{sec:proof-ltwo-lower-bound}

The basic outline of our proofs is similar.  At a high level, for each
binary vector $\packval$ in the Boolean hypercube $\packset =
\{-1,1\}^d$, we construct a linear function $f_\packval$ that is
``well-separated'' from the other functions $\{f_\altpackval,
\altpackval \neq \packval \}$.  Our notion of separation enforces the
following property: if $\optvar^\packval$ minimizes $f_\packval$ over
$\optdomain$, then for each coordinate $j \in [d]$ for which
$\sign(\what{\optvar}_j) \neq \sign(\optvar^\packval_j)$, there is an
additive penalty in the optimization accuracy
$f_\packval(\what{\optvar}) - f_\packval(\optvar^\packval)$.
Consequently, we can lower bound the optimization accuracy by the
testing error in the following \emph{canonical testing problem}:
nature chooses an index $\packval \in \packset$ uniformly at random,
and we must identify the indices $\packval_j$ based on the
observations $\obs[1], \ldots, \obs[\totaliter]$.  By applying lower
bounds on the testing error related to the Assouad and Le Cam
techniques for lower bounding minimax error~\cite{Yu97}, we thus
obtain lower bounds on the optimization error.

In more detail, consider (instantaneous) objective functions of the
form $F(\optvar; \statval) = \<\optvar, \statval\>$.  For each
$\packval \in \packset$, let $P_\packval$ denote the Gaussian
distribution $\normal(\delta \packval, \stddev^2 I_{d \times d})$,
where $\delta > 0$ is a parameter to be chosen, so that
\begin{align*}
f_\packval(\optvar) \defeq \E_{\statprob_\packval}[F(\optvar;
  \statrv)] = \delta \<\optvar, \packval\>.
\end{align*}
For each $\packval \in \packset$, let $\optvar^\packval$
minimize $f_\packval(\optvar)$ over
$\optdomain \defeq \{\optvar \in \R^d \mid \norm{\optvar}_q \le
\radius\}$.  A calculation shows that $\optvar^\packval =
-\radius\,d^{1/q} \,\packval$, so that $\sign(\optvar^\packval_j) =
-\packval_j$.  Next we claim that, for any vector $\what{\optvar} \in
\R^d$,
\begin{equation}
  f_\packval(\what{\optvar}) - f_\packval(\optvar^\packval)
  \ge \frac{1 - 1/q}{d^{1/q}} \delta \radius
  \sum_{j = 1}^d
  \indic{\sign(\what{\optvar}_j) \neq \sign(\optvar^\packval_j)}.
  \label{eqn:v-gap-lower-bound}
\end{equation}
Inequality~\eqref{eqn:v-gap-lower-bound} shows that if it is possible
to optimize well---that is, to find a vector $\what{\optvar}$ with a
relatively small optimality gap---then it is also possible to estimate
the signs of $\packval$.  To establish
inequality~\eqref{eqn:v-gap-lower-bound}, we state a
lemma providing a gap in optimality for solutions of related problems:
\begin{lemma}
  \label{lemma:onevec-optimization}
  For a given integer $i \in [d]$, consider the two optimization
  problems (over $\optvar \in \R^d$)
  \begin{equation*}
    \mbox{(A)} ~
    \begin{array}{rl}
      \minimize & \optvar^\top \onevec \\
      \subjectto & \norm{\optvar}_q \le 1
    \end{array}
    ~~~ \mbox{and} ~~~
    \mbox{(B)} ~
    \begin{array}{rl}
      \minimize & \optvar^\top \onevec \\
      \subjectto & \norm{\optvar}_q \le 1,
      ~ \optvar_j \ge 0 ~ \mbox{for~} j \in [i],
    \end{array}
  \end{equation*}
  with optimal solutions $\optvar^A$ and $\optvar^B$, respectively.
  Then $\<\onevec, \optvar^A\> \le \<\onevec, \optvar^B\> - (1 - 1/q)
  i / d^{1/q}$.
\end{lemma}
\noindent
See Appendix~\ref{sec:proof-onevec-optimization} for a proof.
Returning to inequality~\eqref{eqn:v-gap-lower-bound}, we note that
$f_\packval(\what{\optvar}) - f_\packval(\optvar^\packval) = \delta
\langle \packval, \what{\optvar} - \optvar^\packval\rangle$.  By
symmetry, Lemma~\ref{lemma:onevec-optimization} implies that for every
coordinate $j$ such that $\sign(\what{\optvar}_j) \neq
\sign(\optvar^\packval_j)$, the objective value
$f_\packval(\what{\optvar})$ must be at least a quantity $(1 - 1/q)
\delta \radius / d^{1/q}$ larger than the optimal value
$f_\packval(\optvar^\packval)$, which yields
inequality~\eqref{eqn:v-gap-lower-bound}.

Now we use inequality~\eqref{eqn:v-gap-lower-bound} to give a
probabilistic lower bound. Consider the mixture distribution $\P
\defeq (1/|\packset|) \sum_{\packval \in \packset}
\statprob_\packval$. For any estimator $\what{\optvar}$, we have
\begin{equation*}
  \max_\packval \E_{\statprob_\packval}[f_\packval(\what{\optvar}) -
    f_\packval(\optvar^\packval)] \ge \frac{1}{|\packset|}
  \sum_{\packval \in \packset} \E_{\statprob_\packval}[
    f_\packval(\what{\optvar}) - f_\packval(\optvar^\packval)] \ge
  \frac{1 - 1/q}{d^{1/q}} \, \delta \radius \sum_{j = 1}^d
  \P(\sign(\what{\optvar}_j) \neq -\packrv_j).
\end{equation*}
Consequently, the minimax error is lower bounded as
\begin{equation}
  \mxerr_\totaliter(\fnclass_{\lipobj,2}, \optdomain) \ge \frac{1 -
    1/q}{d^{1/q}} \, \delta \: \radius \: \Big \{ \inf_{\test}
  \sum_{j=1}^d \P(\test_j(\obs[1], \ldots, \obs[\totaliter]) \neq
  \packrv_j) \Big \},
  \label{eqn:testing-lower-bound}
\end{equation}
where $\test$ denotes any testing function mapping from the
observations $\{\obs[\iter]\}_{\iter=1}^\totaliter$ to $\{-1, 1\}^d$.

Next we lower bound the testing error by a total variation distance.
By Le Cam's inequality, for any set $A$ and distributions $P, Q$, we
have $P(A) + Q(A^c) \ge 1 - \tvnorm{P - Q}$. We apply this inequality
to the ``positive $j$th coordinate'' and ``negative $j$th coordinate''
sampling distributions
\begin{equation*}
  \statprob_{+j} \defeq \frac{1}{2^{d-1}} \sum_{\packval \in \packset
    : \packval_j = 1} \statprob_\packval ~~~ \mbox{and} ~~~
  \statprob_{-j} \defeq \frac{1}{2^{d-1}} \sum_{\packval \in \packset
    : \packval_j = -1} \statprob_\packval,
\end{equation*}
corresponding to conditional distributions over $\obs[\iter]$ given
the events $\{\packval_j = 1\}$ or $\{ \packval_j = -1\}$.  Applying
Le Cam's inequality yields
\begin{equation*}
  \P(\test_j(\obs[1:\totaliter]) \neq \packrv_j) = \half
  \statprob_{+j}(\test_j(\obs[1:\totaliter]) \neq 1) + \half
  \statprob_{-j}(\test_j(\obs[1:\totaliter]) \neq -1) \ge \half
  \left(1 - \tvnorm{\statprob_{+j} - \statprob_{-j}}\right).
\end{equation*}
Combined with the upper bound $\sum_{j=1}^d \tvnorm{\statprob_{+j} -
  \statprob_{-j}} \le \sqrt{d} (\sum_{j=1}^d \tvnorm{\statprob_{+j} -
  \statprob_{-j}}^2)^{\half}$ (from the Cauchy-Schwartz inequality), we obtain
\begin{align}
  \mxerr_\totaliter(\fnclass_{\lipobj,2}, \optdomain) & \ge \left(1 -
  \frac{1}{q}\right) \frac{\delta \radius}{2 d^{1/q}} \sum_{j = 1}^d
  \left(1 - \tvnorm{\statprob_{+j} - \statprob_{-j}}\right) \nonumber
  \\ 
& \geq \left(1 - \frac{1}{q}\right) \frac{d^{1 - 1/q} \delta \radius
  }{2} \Bigg(1 - \frac{1}{\sqrt{d}} \bigg(\sum_{j=1}^d
  \tvnorm{\statprob_{+j} - \statprob_{-j}}^2 \bigg)^\half\Bigg).
  \label{eqn:cs-lower-bound}
\end{align}

The remainder of the proof provides sharp enough bounds on $\sum_j
\tvnorm{\statprob_{+j} - \statprob_{-j}}^2$ to leverage
inequality~\eqref{eqn:cs-lower-bound}.  Define the covariance matrix
\begin{equation}
  \label{eqn:compute-covariance}
  \Sigma \defeq \stddev^2 \left[\begin{matrix}
      \ltwo{\optvar}^2 & \<\optvar, \altoptvar\> \\
      \<\optvar, \altoptvar\> & \ltwo{\altoptvar}^2
    \end{matrix} \right]
  = \stddev^2 \left[\optvar ~ \altoptvar\right]^\top
  \left[\optvar ~ \altoptvar \right],
\end{equation}
with the corresponding shorthand $\Sigma^\iter$ for the covariance
computed for the $t^{\text{th}}$ pair $(\optvar[\iter],
\altoptvar[\iter])$. We have:
\begin{lemma}
  \label{lemma:coordinate-tv-bound}
  For each $j \in \{1, \ldots, d\}$, the total variation norm is
  bounded as
  \begin{equation}
    \tvnorm{\statprob_{+j} - \statprob_{-j}}^2
    \le \delta^2 \sum_{\iter = 1}^\totaliter
    \E\left[
      \left[\begin{matrix} \optvar[\iter]_j \\ \altoptvar[\iter]_j
        \end{matrix}\right]^\top
      (\Sigma^\iter)^{-1}
      \left[\begin{matrix} \optvar[\iter]_j \\ \altoptvar[\iter]_j
        \end{matrix}\right]\right].
    \label{eqn:super-j-tv-bound}
  \end{equation}
\end{lemma}
\noindent
See Appendix~\ref{sec:proof-coordinate-tv-bound} for a proof
of this lemma.

Now we use the bound~\eqref{eqn:super-j-tv-bound} to provide a further
lower bound on inequality~\eqref{eqn:cs-lower-bound}. We first note
the identity
\begin{equation*}
  \sum_{j=1}^d
  \left[\begin{matrix} \optvar_j \\ \altoptvar_j
    \end{matrix}\right]
  \left[\begin{matrix} \optvar_j \\ \altoptvar_j
    \end{matrix}\right]^\top
  = \left[\begin{matrix}
      \ltwo{\optvar}^2 & \<\optvar, \altoptvar\> \\
      \<\optvar, \altoptvar\> & \ltwo{\altoptvar}^2 \end{matrix}\right].
\end{equation*}
Recalling the definition~\eqref{eqn:compute-covariance} of the covariance
matrix $\Sigma$, Lemma~\ref{lemma:coordinate-tv-bound} implies that
\begin{align}
  \sum_{j=1}^d
  \tvnorm{\statprob_{+j} - \statprob_{-j}}^2
  & \le \delta^2 \sum_{\iter = 1}^\totaliter
  \E\bigg[\sum_{j=1}^d \tr\bigg(
    (\Sigma^\iter)^{-1}
    \left[\begin{matrix} \optvar[\iter]_j \\ \altoptvar[\iter]_j
      \end{matrix}\right]\left[\begin{matrix}
        \optvar[\iter]_j \\ \altoptvar[\iter]_j
      \end{matrix}\right]^\top
    \bigg)\bigg] \nonumber \\
  & = \frac{\delta^2 }{\stddev^2}
  \sum_{\iter = 1}^\totaliter \E\left[\tr\left((\Sigma^\iter)^{-1}
    \Sigma^\iter\right)\right]
  = 2 \frac{\totaliter \delta^2}{\stddev^2}.
  \label{eqn:summed-tv-q-norms}
\end{align}
Returning to the estimation lower bound~\eqref{eqn:cs-lower-bound},
we thus find the nearly final lower bound
\begin{equation}
  \mxerr_\totaliter(\fnclass_{\lipobj,2}, \optdomain)
  \ge \left(1 - \frac{1}{q}\right)
  \frac{d^{1 - 1/q}\delta \radius}{2} \bigg(1 -
  \left(\frac{2\totaliter \delta^2}{d \stddev^2}\right)^\half
  \bigg).
  \label{eqn:q-key-inequality}
\end{equation}

Enforcing $(F, \statprob) \in \fnclass_{\lipobj,2}$ amounts to choosing the
parameters $\stddev^2$ and $\delta^2$ so that $\E[\ltwo{\statrv}^2] \leq
\lipobj^2$ for $\statrv \sim \normal(\delta\cubeval, \stddev^2 I_{d \times
  d})$, after which we may use inequality~\eqref{eqn:q-key-inequality} to
complete the proof of the lower bound. By construction, we have
$\E[\ltwo{\statrv}^2] = (\delta^2 + \stddev^2) \dim$, so choosing $\stddev^2 =
8 \lipobj^2 / 9 \dim$ and $\delta^2 = (\lipobj^2 / 9)\min\{1/\totaliter,
1/\dim\}$ guarantees that
\begin{equation*}
  1 - \left(\frac{2 \totaliter \delta^2}{d \stddev^2}\right)^\half
  \ge 1 - \left(\frac{18}{72}\right)^\half
  = \half
  ~~~ \mbox{and} ~~~
  \E[\ltwo{\statrv}^2] = \frac{8 \lipobj^2}{9}
  + \frac{\lipobj^2 \dim}{9} \min\left\{\frac{1}{\totaliter},
  \frac{1}{\dim}\right\} \le \lipobj^2.
\end{equation*}
Substituting these choices of $\delta$ and $\stddev^2$ in
inequality~\eqref{eqn:q-key-inequality}
gives the lower bound
\begin{equation*}
  \mxerr_\totaliter(\fnclass_{\lipobj,2}, \optdomain)
  \ge \frac{1}{12} \left(1 - \frac{1}{q}\right)
  d^{1 - 1/q} \radius \lipobj
  \min\left\{\frac{1}{\sqrt{\totaliter}}, \frac{1}{\sqrt{d}}\right\}
  = \frac{1}{12} \left(1 - \frac{1}{q}\right)
  \frac{d^{1 - 1/q} \radius \lipobj}{\sqrt{\totaliter}}
  \min\left\{1, \sqrt{\totaliter / d}\right\}.
\end{equation*}

To complete the proof of the claim~\eqref{eqn:ltwo-lower-bound}, we note that
the above lower bound also applies to any $d_0$-dimensional problem for
$d_0 \le d$. More rigorously, we choose $\packset = \{-1, 1\}^{d_0} \times
\{0\}^{d-d_0}$, and define the sampling distribution $\statprob_\packval$ on
$\statrv$ so that given $\packval \in \packset$, the coordinate distributions
of $\statrv$ are independent with $\statrv_j \sim \normal(\delta \packval_j,
\stddev^2)$ for $j \le d_0$ and $\statrv_j = 0$ for $j > d_0$. A reproduction
of the preceding proof, substituting $d_0 \le d$ for each appearance of the
dimension $d$, then yields
the claimed bound~\eqref{eqn:ltwo-lower-bound}
when we choose $d_0 = \min\{\totaliter, d\}$.

\paragraph{Remarks on multiple evaluations:}
By an extension of Lemma~\ref{lemma:coordinate-tv-bound}, we may consider
the case in which at each iteration, the method may
query for function values at the $\numobs$ points \mbox{$\optvar_{(1)},
\ldots, \optvar_{(\numobs)} \in \R^d$}. Let
$\optvar[\iter]_{j,(i)}$ denote the $j$th coordinate of the $i$th query
point in iteration $\iter$. In this case, an immediate analogue of
Lemma~\ref{lemma:coordinate-tv-bound} implies
\begin{equation*}
  \tvnorm{\statprob_{+j} - \statprob_{-j}}^2
  \le \delta^2 \sum_{\iter=1}^\totaliter
  \E\left[\begin{matrix} \optvar[\iter]_{j,(1)} \\
        \vdots \\ \optvar[\iter]_{j,(\numobs)} \end{matrix}\right]^\top
    (\Sigma^\iter)^{-1}
    \left[\begin{matrix} \optvar[\iter]_{j,(1)} \\
        \vdots \\ \optvar[\iter]_{j,(\numobs)} \end{matrix}\right],
\end{equation*}
where $\Sigma^\iter = \stddev^2 [\optvar[\iter]_{(1)} ~ \cdots ~
\optvar[\iter]_{(\numobs)}]^\top[\optvar[\iter]_{(1)} ~ \cdots ~
\optvar[\iter]_{(\numobs)}]$ denotes a covariance matrix as in
equation~\eqref{eqn:compute-covariance}. Following the calculation
of inequality~\eqref{eqn:summed-tv-q-norms}, we obtain
\begin{equation*}
  \sum_{j=1}^d
  \tvnorm{\statprob_{+j} - \statprob_{-j}}^2
  \le \frac{\delta^2}{\stddev^2} \sum_{\iter=1}^\totaliter
  \E\left[\tr\left((\Sigma^\iter)^{-1} \Sigma^\iter\right)\right]
  = \frac{\numobs \totaliter \delta^2}{\stddev^2}.
\end{equation*}
Substituting this inequality in place of~\eqref{eqn:summed-tv-q-norms}
and following the subsequent proof implies the lower 
bound $\frac{1}{10} (1 - q^{-1})d^{1 - 1/q} \radius \lipobj
/ \sqrt{\numobs \totaliter} \cdot \min\{1, \sqrt{\totaliter / d}\}$.
Replacing $d$ with $\min\{\totaliter, d\}$
gives inequality~\eqref{eqn:multi-point-lower-bound}.


\subsection{Proof of Proposition~\ref{proposition:lone-lower-bound}}
\label{sec:proof-lone-lower-bound}

The proof is similar to that of Proposition~\ref{proposition:ltwo-lower-bound},
except instead of using the set $\packset = \{-1,1\}^d$, we use the $2d$
standard basis vectors and their negatives, that is, $\packset = \{\pm
e_j\}_{j=1}^d$. We use the same sampling distributions as in the proof of
Proposition~\ref{proposition:ltwo-lower-bound}, so under $\statprob_\packval$
the random vectors $\statrv \sim \normal(\delta \packval, \stddev^2 I_{d
  \times d})$, and we have $f_\packval = \E_{\statprob_\packval}[F(\optvar;
  \statrv)] = \delta\<\optvar, \packval\>$.  Let us define $\statprob_j$ to be
the distribution $\statprob_\packval$ for $\packval = e_j$ and similarly for
$\statprob_{-j}$, and let $\optvar^\packval = \argmin_{\optvar} \{
f_\packval(\optvar) \mid \lone{\optvar} \le \radius\} = - \radius \packval$.

We now provide the reduction from optimization to testing.  First,
if $\packval = \pm e_j$, then any estimator $\what{\optvar}$ satisfying
$\sign(\what{\optvar}_j) \neq \sign(\optvar^\packval_j)$ must have
$f_\packval(\what{\optvar}) - f_\packval(\optvar^\packval) \ge \delta
\radius$.  We thus see that for $\packval \in \{\pm e_j\}$,
\begin{equation*}
  f_\packval(\what{\optvar}) - f_\packval(\optvar^\packval)
  \ge \delta \, \radius \; \indic{\sign(\what{\optvar}_j) \neq
    \sign(\optvar^\packval_j)}.
\end{equation*}
Consequently,
we obtain the multiple binary hypothesis testing lower bound
\begin{align*}
  \lefteqn{\max_{\packval} \E_{\statprob_\packval}[f_\packval(\what{\optvar})
      - f_\packval(\optvar^\packval)]
    \ge \frac{1}{2d} \sum_{\packval \in \packset}
    \E_{\statprob_\packval}[f_\packval(\what{\optvar})
      - f_\packval(\optvar^\packval)]} \\
  & \qquad\qquad\quad ~ \ge \frac{\delta \radius}{2d} \sum_{j=1}^d
  \left[\statprob_j(\sign(\what{\optvar}_j) \neq -1)
    + \statprob_{-j}(\sign(\what{\optvar}_j) \neq 1)\right]
  \stackrel{(i)}{\ge} \frac{\delta \radius}{2d} \sum_{j = 1}^d
  \left[1 - \tvnorm{\statprob_j - \statprob_{-j}}\right].
\end{align*}
For the final inequality~$(i)$, we applied Le Cam's inequality
as in the proof of Proposition~\ref{proposition:ltwo-lower-bound}. Thus,
as in
the derivation of inequality~\eqref{eqn:cs-lower-bound} from the
Cauchy-Schwarz inequality, this yields
\begin{equation}
  \mxerr_\totaliter(\fnclass_{\lipobj,\infty}, \optdomain) \ge
  \max_{\packval} \E_{\statprob_\packval}[f_\packval(\what{\optvar})
    - f_\packval(\optvar^\packval)]
  \ge \frac{\delta \radius}{2}
  \Bigg(1 - \frac{1}{\sqrt{d}} \bigg(\sum_{j=1}^d \tvnorm{\statprob_j
    - \statprob_{-j}}^2\bigg)^\half \Bigg).
  \label{eqn:l1-cs-lower-bound}
\end{equation}

We now turn to providing a bound on $\sum_{j=1}^d \tvnorm{\statprob_j
  - \statprob_{-j}}^2$ analogous to that in the proof of
Proposition~\ref{proposition:ltwo-lower-bound}. We claim that
\begin{equation}
  \label{eqn:l1-tv-sum-bound}
  \sum_{j=1}^d \tvnorm{\statprob_j - \statprob_{-j}}^2
  \le 2 \frac{\totaliter \delta^2}{\stddev^2}.
\end{equation}
Inequality~\eqref{eqn:l1-tv-sum-bound} is nearly immediate from
Lemma~\ref{lemma:coordinate-tv-bound}.
Indeed, given the pair $\pairoptvar = [\optvar ~ \altoptvar] \in \R^{d \times
  2}$, the observation $\obs = \pairoptvar^\top \statrv$ is distributed
(conditional on $\packval$ and $\pairoptvar$) as $\normal(\delta
\pairoptvar^\top \packval, \Sigma)$ where $\Sigma = \stddev^2 \pairoptvar^\top
\pairoptvar$ is the covariance~\eqref{eqn:compute-covariance}.
For $\packval = e_j$ and $\altpackval = -e_j$, we know that
$\<\optvar, \packval - \altpackval\> = 2 \optvar_j$ and so
\begin{equation*}
  \dkl{\normal(\delta \pairoptvar^\top \packval, \Sigma)}{\normal(\delta
    \pairoptvar^\top \altpackval, \Sigma)}
  = 2 \delta^2
  \left[\begin{matrix} \optvar_j \\ \altoptvar_j \end{matrix}\right]^\top
  \Sigma^{-1}\left[\begin{matrix} \optvar_j \\ \altoptvar_j
    \end{matrix}\right].
\end{equation*}
By analogy with the proof of Lemma~\ref{lemma:coordinate-tv-bound}, we may
repeat the derivation of inequalities~\eqref{eqn:super-j-tv-bound}
and~\eqref{eqn:summed-tv-q-norms} \emph{mutatis mutandis} to obtain
inequality~\eqref{eqn:l1-tv-sum-bound}.
%
Combining inequalities~\eqref{eqn:l1-cs-lower-bound}
and~\eqref{eqn:l1-tv-sum-bound} then gives the lower bound
\begin{equation*}
  \mxerr_\totaliter(\fnclass_{\lipobj,\infty}, \optdomain)
  \ge \frac{\delta \radius}{2}
  \left(1 - \bigg(\frac{2 \delta^2 \totaliter}{d \stddev^2}\bigg)^\half
  \right).
\end{equation*}

It thus remains to choose $\delta$ and $\stddev^2$ to guarantee the
containment $(F, \statprob) \in \fnclass_{\lipobj,\infty}$.
Equivalently, we must establish the gradient bound
$\E[\linf{\statrv}^2] \le \lipobj^2$, with which the next lemma helps.
\begin{lemma}
  \label{lemma:normal-linf-norm}
  Given any vector with $\linf{\cubeval} \le 1$, and the random vector
  $\statrv \sim \normal(\delta \cubeval, \stddev^2 I_{d \times d})$, we
  have
  \begin{equation*}
    \E[\linf{\statrv}^2] \le
    3 \stddev^2 \log(3 \dim) + 4 \delta^2.
  \end{equation*}
\end{lemma}
\begin{proof}
  The vector $Z = \statrv - \delta \cubeval$ has $\normal(0, \stddev^2
  I_{d \times d})$ distribution.
  By Jensen's inequality, for all $\epsilon \ge 0$ we have
  \begin{equation*}
    \linf{\statrv}^2 \le (1 + \epsilon) \linf{Z}^2 + (1 + \epsilon^{-1})
    \delta^2 \linf{\cubeval}^2
    \le (1 + \epsilon) \linf{Z}^2 + (1 + \epsilon^{-1}) \delta^2.
  \end{equation*}
  Classical results on Gaussian vectors~\cite[Chapter
    2]{BuldyginKo00} imply
  $\E[\linf{Z}^2] \le \stddev^2 (\frac{1}{\lambda} \log d + 
  \frac{1}{2\lambda} \log \frac{1}{1 - 2\lambda})$ for all
  $\lambda \in [0, 1/2]$, so taking
  $\epsilon = 1/3$ and $\lambda = 4/9$ implies the lemma.
\end{proof}
\noindent
As a consequence of Lemma~\ref{lemma:normal-linf-norm}, by taking
\begin{equation*}
  \stddev^2 = \frac{2 \lipobj^2}{9\log(3d)}
  ~~~ \mbox{and} ~~~
  \delta^2 = \frac{\lipobj^2}{36 \log(3d)} \min \left\{1, \frac{d}{\totaliter}
  \right\},
\end{equation*}
we obtain the bounds
\begin{equation*}
  \E[\linf{\statrv}^2] \le \frac{2}{3} \lipobj^2
  + \frac{4}{36} \lipobj^2 < \lipobj^2
  ~~~ \mbox{and} ~~~
  1 - \left(\frac{2 \delta^2 \totaliter}{d \stddev^2}\right)^\half
  \ge 1 - \left(\frac{18}{72}\right)^\half = \half.
\end{equation*}
Substituting into the lower
bound on $\mxerr_\totaliter$ yields
\begin{equation*}
  \mxerr_\totaliter(\fnclass_{\lipobj,\infty}, \optdomain)
  \ge 
  \frac{\delta \radius}{4}
  \ge \frac{1}{24 \sqrt{\log(3d)}}
  \, \frac{\lipobj \radius}{\sqrt{\totaliter}}
  \min\left\{\sqrt{\totaliter}, \sqrt{d}\right\}.
\end{equation*}
Modulo this lower bound holding for each dimension $d_0 \le d$, this
completes the proof.

To complete the proof, we note that as in the proof of
Proposition~\ref{proposition:ltwo-lower-bound}, we may provide a lower bound
on the optimization error for any $d_0 \le d$-dimensional problem. In
particular fix $d_0 \le d$ and let $\packset = \{\pm e_j\}_{j=1}^{d_0} \subset
\R^d$. Now, conditional on $\packval \in \packset$, let
$\statprob_\packval$ denote the distribution on $\statrv$ with
independent coordinates whose distributions are
$\statrv_j \sim \normal(\delta \packval_j, \stddev^2)$ for
$j \le d_0$ and $\statrv_j = 0$ for $j > d_0$.
As in the proof Proposition~\ref{proposition:ltwo-lower-bound}, 
we may reproduce the preceding arguments by substituting $d_0 \le d$
for every appearance of the dimension $d$,
giving that for all $d_0 \le d$,
\begin{equation*}
  \mxerr_\totaliter(\fnclass_{\lipobj,\infty}, \optdomain)
  \ge \frac{1}{24 \sqrt{\log (3 d_0)}}
  \, \frac{\lipobj \radius}{\sqrt{\totaliter}}
  \min\left\{\sqrt{\totaliter}, \sqrt{d_0}\right\}.
\end{equation*}
Choosing $d_0 = \min\{d, \totaliter\}$ completes
the proof of Proposition~\ref{proposition:lone-lower-bound}.


\subsection{Proof of Proposition~\ref{proposition:single-observation}}

This proof is somewhat similar to that of
Proposition~\ref{proposition:ltwo-lower-bound}, in that we use the set
$\packset = \{-1, 1\}^d$ to construct a collection of functions whose minima
are relatively well-separated, but for which function evaluations are hard
to distinguish. In particular, for $\delta > 0$, we construct functions
$f_\packval$ whose minima---for different elements $\packval,
\altpackval$---are all of the order $\delta \lone{\packval - \altpackval}$
distant from one another, yet $\sup_{\optvar} |f_\packval(\optvar) -
f_\altpackval(\optvar)| \lesssim \delta$, so that many observations are
necessary to distinguish the functions.

In more detail, for $\packval \in\packset = \{-1, 1\}^d$, define the
probability distribution $\statprob_\packval$ to be supported on
$\{\packval\} \times \R$, where each independent draw $\statrv = (\packval,
\normalrv) \sim \statprob_\packval$ contains an independent $\normalrv \sim
\normal(0, \stddev^2)$. Fix $\delta \in \openleft{0}{\radius d^{-1/q}}$,
and define $\lipobj_d = \lipobj d^{-1/p}$.
Then for
$\statval = (\packval, \normalrv)$, we define
\begin{equation*}
  F(\optvar; \statval)
  = \lipobj_d \lone{\optvar - \delta \packval} + \normalrv,
  ~~~ \mbox{so} ~~~
  f_\packval(\optvar) = \lipobj_d \lone{\optvar - \delta\packval}
  ~~~ \mbox{and} ~~~
  F(\optvar; (\packval, \normalrv)) = f_\packval(\optvar) + \normalrv.
\end{equation*}
Consequently, we have $\delta \packval = \optvar^\packval \defeq
\argmin_{\optvar \in \optdomain} f_\packval(\optvar)$, as
$\norm{\delta \packval}_q \le \radius d^{-1/q} \norm{\packval}_q
= \radius$, and
the variance bound
$\E[(F(\optvar; \statrv) - f(\optvar))^2] \le \stddev^2$ is evident.
Moreover, we have
\begin{equation*}
  \norm{\partial F(\optvar; \statval)}_p
  = \lipobj_d \norm{\sign(\optvar - \delta \packval)}_p
  \le \lipobj d^{-1/p} d^{1/p} = \lipobj,
\end{equation*}
so the functions belong to $\fnclass_{\stddev,\lipobj,p}$.  By inspection,
we have the separation
\begin{equation*}
  f_\packval(\optvar) - f_\packval(\optvar^\packval)
  \ge \delta \lipobj_d \sum_{j=1}^d \indic{\sign(\optvar_j) \neq \packval_j},
\end{equation*}
which is analogous to inequality~\eqref{eqn:v-gap-lower-bound}.

Abusing notation and defining $\statprob_\packval$ to be the distribution
of the $\totaliter$ observations $F(\optvar[\iter]; \statrv[\iter])$
available to the method, our earlier extension~\eqref{eqn:cs-lower-bound}
of Assouad's method implies
\begin{equation}
  \label{eqn:assouad-single-evaluation}
  \mxerrone_\totaliter(\fnclass_\stddev, \optdomain)
  \ge \frac{\delta \lipobj_d}{2}
  \sum_{j=1}^d \left(1 - \tvnorm{\statprob_{+j} - \statprob_{-j}}\right)
  \ge \frac{d \delta \lipobj_d}{2}
  \Bigg(1 - \bigg(\frac{1}{2d} \sum_{j=1}^d
  \dkl{\statprob_{+j}}{\statprob_{-j}}\bigg)^\half\Bigg),
\end{equation}
where $\statprob_{+j} = 2^{1-d} \sum_{\packval : \packval_j = 1}
\statprob_\packval$, and similarly for $-j$.
Now, note that for
any $\packval, \altpackval \in \packset$ such that
$\lone{\packval - \altpackval} \le 2$, we have
the inequality
\begin{equation*}
  \sup_{\optvar} |f_\packval(\optvar) - f_\altpackval(\optvar)|
  \le \lipobj_d \lone{\delta \packval - \delta \altpackval}
  \le 2 \delta \lipobj_d
\end{equation*}
(compare with Lemma 10 of \citet{Shamir13}).  In particular, this uniform
inequality implies that for distributions $\statprob_\packval$ and
$\statprob_\altpackval$, the observations $F(\optvar; \statrv) =
f_\packval(\optvar) + \normalrv$ are normally distributed random
variables with (absolute) difference in means bounded by $\delta \lipobj_d
\lone{\packval - \altpackval}$ and variance $\stddev^2$.  Using that the KL
divergence is jointly convex in both its arguments, we have (by a completely
parallel argument to the proof of Lemma~\ref{lemma:coordinate-tv-bound} in
Appendix~\ref{sec:proof-coordinate-tv-bound}) that
\begin{align*}
  \dkl{\statprob_{+j}}{\statprob_{-j}}
  & \le \frac{1}{2^d}
  \sum_{\packval \in \packset}
  \dkl{\statprob_{\packval,+j}}{\statprob_{\packval,-j}} \\
  & \le \frac{\totaliter}{2^d}
  \sum_{\packval \in \packset}
  \dkl{\normal(\delta \lipobj_d, \stddev^2)}{
    \normal(-\delta \lipobj_d, \stddev^2)}
  = \frac{\totaliter}{2^d}
  \sum_{\packval \in \packset}
  \frac{1}{2 \stddev^2} 4 \lipobj_d^2 \delta^2
  = \frac{2 \totaliter \lipobj_d^2 \delta^2}{\stddev^2}.
\end{align*}
Substituting the KL divergence bound in the preceding display into our
inequality~\eqref{eqn:assouad-single-evaluation}, we find
\begin{equation*}
  \mxerrone_\totaliter(\fnclass_\stddev, \optdomain)
  \ge \frac{d \delta \lipobj_d}{2}
  \left(1 - \sqrt{
    \totaliter \frac{\lipobj_d^2 \delta^2}{\stddev^2}}
  \right)
  = \frac{d \delta \lipobj_d}{2}
  \left(1 - \delta \frac{\sqrt{k} \lipobj_d}{\stddev}\right).
\end{equation*}
Choosing $\delta
= \min\{\radius d^{-1/q}, \stddev / 2 \lipobj_d \sqrt{\totaliter}\}$
and substituting $\lipobj_d = \lipobj d^{-1/p}$
gives the proposition.



\section{Discussion}

We have analyzed algorithms for optimization problems that use only random
function values---as opposed to gradient computations---to minimize an
objective function.  The algorithms we present are optimal: their convergence
rates cannot be improved (in a minimax sense) by more than numerical constant
factors. In addition to showing the optimality of several algorithms for
smooth convex optimization without gradient information, we have also shown
that the non-smooth case is no more difficult from an iteration complexity
standpoint, though it requires more carefully constructed randomization
schemes.  As a consequence of our results, we have additionally
 attained sharp rates
for bandit online convex optimization problems with multi-point feedback.  We
have also shown the necessary transition in convergence rates between
gradient-based algorithms and those that compute only function values: when
(sub)gradient information is available, attaining $\epsilon$-accurate solution
to an optimization problem requires $\order(1 / \epsilon^2)$ gradient
observations, while at least $\Omega(d / \epsilon^2)$ observations---but no
more---are necessary using paired function evaluations, and at 
least $\Omega(d^2 / \epsilon^2)$ are necessary using only a single
function evaluation. An
interesting open question is to further understand this
last setting: what is the optimal iteration complexity in this case?



\paragraph{Acknowledgments}

This material supported in part by ONR MURI grant N00014-11-1-0688 and
U.S.\ Army Research Office under grant no.\ W911NF-11-1-0391, as well
as by NSF grant CIF-31712-23800.  JCD was also supported by an NDSEG
fellowship and a Facebook PhD fellowship.  We also thank the editor,
Nicol\'o Cesa-Bianchi, and two reviewers for multiple constructive
suggestions.


\appendix

\section{Technical results for convergence arguments}

In this appendix, we collect the proofs of the various lemmas used in our
convergence arguments.


\subsection{Proof of Lemma~\ref{lemma:vector-moments}}
\label{appendix:proof-vector-moments}

We consider each of the distributions in turn.
When $\perturbrv$ has $\normal(0, I_{d \times d})$ distribution,
standard $\chi^2$-distributed random variable calculations
imply
\begin{equation*}
  \E\left[\ltwo{\perturbrv}^k\right]
  = 2^\frac{k}{2} \frac{\Gamma(\frac{k}{2} + \frac{d}{2})}{
    \Gamma(\frac{d}{2})}.
\end{equation*}
That $\E[\perturbrv \perturbrv^\top] = I_{d \times d}$ is immediate, and the
constant values $c_k$ for $k \le 4$ follow from direct calculations.  For
samples $\perturbrv$ from the $\ell_2$-sphere, it is clear that
$\ltwo{\perturbrv} = \sqrt{d}$, so we may take $c_k = 1$ in the statement of
the lemma. When $\perturbrv \sim \uniform(\B^d)$, the density $p(t)$ of
$\ltwo{\perturbrv}$ is given by $\dim \cdot t^{\dim-1}$; consequently, for any
$k > -\dim$ we have
\begin{equation}
  \label{eqn:moment-ball}
  \E[\ltwo{\perturbrv}^k]
  = \int_0^1 t^k p(t) \, dt
  = \dim \int_0^1 t^{\dim+k-1} \, dt
  = \frac{\dim}{\dim+k}.
\end{equation}
Thus for $\perturbrv \sim
\uniform(\sqrt{\dim+2}\,\B^d)$ we have
$\E[\perturbrv \perturbrv^\top] = I_{\dim \times \dim}$,
and $\E[\ltwo{\perturbrv}^k] = (d + 2)^{k/2} d / (d + k)$.


\subsection{Proof of Lemma~\ref{lemma:moment-bound}}
\label{appendix:lipschitz-moment-bounds}

\newcommand{\difffunc}{\Delta} 
\newcommand{\event}{A_{\optvar}}

The proof of Lemma~\ref{lemma:moment-bound} is based on a sequence of
auxiliary results.  Since the Lipschitz continuity of $h$ implies the
result for $\dim = 1$ directly, we focus on the case $\dim \geq 2$.
First, we have the following standard result on the
dimension-independent concentration of rotationally symmetric
sub-Gaussian random vectors. We use this to prove that the perturbed
$h$ is close to the unperturbed $h$ with high probability.

\begin{lemma}[Rotationally invariant concentration]
  \label{lemma:rotation-concentration}
  Let $\perturbrv$ be a random variable in $\R^d$ having one of the following
  distributions: $\normal(0,I_{\dim \times \dim})$,
  $\uniform(\sqrt{\dim+2}\,\B^d)$, or
  $\uniform(\sqrt{\dim}\,\S^{\dim-1})$. There is a universal (numerical)
  constant $c > 0$ such that for any $\lipobj$-Lipschitz continuous function
  $h$,
  \begin{equation*}
    \P\left(|h(\perturbrv)
        - \E[h(\perturbrv)]|
        > \epsilon \right)
    \le 2 \exp\left(-\frac{c\, \epsilon^2}{\lipobj^2} \right).
  \end{equation*}
  In the case of the normal distribution, we may take $c = \half$.
\end{lemma}
\noindent
These results are standard (e.g., see Propositions~1.10 and 2.9 of
\citet{Ledoux01}). \\

Our next result shows that integrating out $\perturbrv_2$ leaves us
with a smoother deviation problem, at the expense of terms of order at
most $\smoothparam^k \log^{k/2}(\dim)$. To state the lemma, we define
the difference function $\difffunc_\smoothparam(\optvar) =
\E[h(\optvar + \smoothparam \perturbrv_2)] - h(\optvar)$. Note that
since $h$ is convex and $\E[\perturbrv_2] = 0$, Jensen's inequality
implies $\difffunc_\smoothparam(\optvar) \ge 0$.

\begin{lemma}
  \label{lemma:concentration-h}
  Under the conditions of Lemma~\ref{lemma:moment-bound}, we have
  \begin{equation*}
    \E\left[\left|h(\perturbrv_1 + \smoothparam \perturbrv_2) -
      h(\perturbrv_1)\right|^k\right]
    \le 2^{k-1} \E[\difffunc_\smoothparam(\perturbrv_1)^k]
      + c^{-\frac{k}{2}} 2^{k-1} k^{\frac{k}{2}}
        \smoothparam^k \log^{\frac{k}{2}} (\dim+2k)
      + \sqrt{2} \smoothparam^k
  \end{equation*}
  for any $k \ge 1$.  Here $c$ is the same constant in
  Lemma~\ref{lemma:rotation-concentration}.
\end{lemma}
\begin{proof}
  \renewcommand{\complement}{c} For each $\optvar \in \optdomain$, the
  function $\altoptvar \mapsto h(\optvar + \smoothparam \altoptvar)$
  is $\smoothparam$-Lipschitz, so that
  Lemma~\ref{lemma:rotation-concentration} implies
  that
\begin{equation*} 
\P \left( \big|h(\optvar + \smoothparam \perturbrv_2) - \E[h(\optvar +
  \smoothparam \perturbrv_2)] \big| > \epsilon \right) \leq 2\exp
\left(-\frac{c\,\epsilon^2}{\smoothparam^2}\right).
  \end{equation*}
On the event $\event(\epsilon) \defeq \{|h(\optvar + \smoothparam
\perturbrv_2) - \E[h(\optvar + \smoothparam \perturbrv_2)]| \le
\epsilon\}$, we have
  \begin{equation*}
    \left|h(\optvar + \smoothparam \perturbrv_2)-h(\optvar)\right|^k
    \le 2^{k-1}\left|h(\optvar + \smoothparam \perturbrv_2)
        - \E[h(\optvar + \smoothparam \perturbrv_2)\right|^k
        + 2^{k-1} \difffunc_\smoothparam(\optvar)^k
    \le 2^{k-1} \epsilon^k + 2^{k-1} \difffunc_\smoothparam(\optvar)^k,
  \end{equation*}
  which implies
  \begin{subequations}
    \begin{equation}
      \E\left[\left|h(\optvar + \smoothparam \perturbrv_2)-h(\optvar)\right|^k
        \cdot \indic{\event(\epsilon)}\right]
      \le 2^{k-1} \difffunc_\smoothparam(\optvar)^k + 2^{k-1} \epsilon^k.
      \label{eqn:bound-event}
    \end{equation}
    On the complement $\event^\complement(\epsilon)$, which occurs with
    probability at most $2\exp(-c\epsilon^2/\smoothparam^2)$, we use the
    Lipschitz continuity of $h$ and Cauchy-Schwarz inequality to obtain
    \begin{equation*}
      \E\left[\left|h(\optvar + \smoothparam \perturbrv_2)-h(\optvar)\right|^k
        \cdot \indic{\event(\epsilon)^\complement}\right]
      \le \E\left[\smoothparam^k\ltwo{\perturbrv_2}^k \cdot
        \indic{\event(\epsilon)^\complement} \right]
      \le \smoothparam^k \E[\ltwo{\perturbrv_2}^{2k}]^{\half} \cdot
      \P\left(\event(\epsilon)^\complement\right)^{\half}.
    \end{equation*}
    By direct calculations, Assumption~\ref{assumption:smoothing-dist-general}
    implies that $\E[\ltwo{\perturbrv_2}^{2k}] \le (d + 2k)^k$.
    %
    %
    Thus,
    \begin{equation}
      \E\left[\left|h(\optvar + \smoothparam \perturbrv_2)-h(\optvar)\right|^k
        \cdot \indic{\event(\epsilon)^\complement}\right]
      \le \smoothparam^k (\dim+2k)^{\frac{k}{2}} \cdot
      \sqrt{2} \exp\left(-\frac{c\,\epsilon^2}{2\smoothparam^2}\right).
      \label{eqn:bound-event-complement}
    \end{equation}
  \end{subequations}
  Combining the estimates~\eqref{eqn:bound-event}
  and~\eqref{eqn:bound-event-complement} gives
  \begin{equation*}
    \E\left[\left|h(\optvar + \smoothparam \perturbrv_2)-h(\optvar)\right|^k
      \right]
    \le 2^{k-1} \difffunc_\smoothparam(\optvar)^k
      + 2^{k-1} \epsilon^k
      + \sqrt{2} \smoothparam^k (\dim+2k)^{\frac{k}{2}}
        \exp\left(-\frac{c\,\epsilon^2}{2\smoothparam^2}\right).
  \end{equation*}
  Setting $\epsilon^2 = \frac{k}{c} \smoothparam^2 \log(\dim+2k)$ and
  taking expectations over $\perturbrv_1 \sim \smoothingdist_1$
  gives Lemma~\ref{lemma:concentration-h}.
\end{proof}

By Lemma~\ref{lemma:concentration-h}, it suffices to control the bias
$\E[\difffunc_\smoothparam(\perturbrv_1)] = \E[h(\perturbrv_1 + \smoothparam
  \perturbrv_2) - h(\perturbrv_1)]$. The following result allows us to
reduce this problem to one of bounding a certain one-dimensional expectation.

\begin{lemma}
  \label{lemma:linear-to-absolute-value}
  Let $\perturbrv$ and $\altperturbrv$ be random variables in $\R^\dim$
  with rotationally invariant distributions and finite first moments.
  Let $\mc{H}$ denote the set of $1$-Lipschitz convex functions
  $h \colon \R^\dim \to \R$, and for $h \in \mc{H}$, define
  $V(h) = \E[h(\altperturbrv) - h(\perturbrv)]$. Then
  \begin{equation*}
     \sup_{h \in \mc{H}} V(h) = \sup_{a \in \R_+}
      \E\left[|\ltwo{\altperturbrv} - a| - |\ltwo{\perturbrv} - a|\right].
  \end{equation*}
\end{lemma}
\begin{proof}
  First, we note that $V(h) = V(h \circ U)$ for any
  unitary transformation $U$; since $V$ is linear,
  if we define $\hat{h}$ as the average of $h \circ U$ over all unitary
  $U$ then $V(h) = V(\hat{h})$.
  Moreover, for $h \in \mc{H}$, we have $\hat{h}(\optvar)
  = \hat{h}_1(\ltwo{\optvar})$ for some $\hat{h}_1 : \R_+ \to \R$, which is
  necessarily $1$-Lipschitz and convex.

  Letting $\mc{H}_1$ denote the $1$-Lipschitz convex $h : \R \to \R$
  satisfying $h(0) = 0$, we thus have $\sup_{h \in \mc{H}} V(h) =
  \sup_{h \in \mc{H}_1} \E[h(\ltwo{\altperturbrv}) -
    h(\ltwo{\perturbrv})]$.  Now, we define $\mc{G}_1$ to be the set
  of measurable non-decreasing functions bounded in $[-1, 1]$. Then by
  known properties of convex functions~\cite{HiriartUrrutyLe96ab}, for
  any $h \in \mc{H}_1$, we can write $h(t) = \int_0^t g(s) ds$ for
  some $g \in \mc{G}_1$.  Using this representation, we have
\begin{align}
\sup_{h \in \mc{H}} V(h) & = \sup_{h \in \mc{H}_1}
\left\{\E[h(\ltwo{\altperturbrv}) - h(\ltwo{\perturbrv})]\right\}
\nonumber \\ 
& = \sup_{g \in \mc{G}_1} \left\{\E[h(\ltwo{\altperturbrv}) -
  h(\ltwo{\perturbrv})], ~ \mbox{where} ~ h(t) = \int_0^t g(s) ds
\right \}.
\label{eqn:derivative-suprema}
\end{align}

  Let $g_a$ denote the $\{-1, 1\}$-valued function with step at $a$,
  that is, $g_a(t) = -\indic{t \le a} + \indic{t > a}$. We define
  $\mc{G}_1^{(n)}$ to be the set of non-decreasing step functions
  bounded in $[-1, 1]$ with at most $n$ steps, that is, functions of
  the form $g(t) = \sum_{i = 1}^n b_i g_{a_i}(t)$, where $|g(t)| \le
  1$ for all $t \in \R$. We may then further simplify the
  expression~\eqref{eqn:derivative-suprema} by replacing $\mc{G}_1$
  with $\mc{G}_1^{(n)}$, that is,
  \begin{equation*}
    \sup_{h \in \mc{H}} V(h)
    = \sup_{n \in \N}
    \sup_{g \in \mc{G}_1^{(n)}}
    \left\{\E[h(\ltwo{\altperturbrv}) - h(\ltwo{\perturbrv})],
    ~ \mbox{where} ~ h(t) = \int_0^t g(s) ds\right\}.
  \end{equation*}
  The extremal points of $\mc{G}_1^{(n)}$ are the step functions $\{g_a \mid a
  \in \R\}$, and since the supremum~\eqref{eqn:derivative-suprema} is linear
  in $g$, it may be taken over such $g_a$.
  Lemma~\ref{lemma:linear-to-absolute-value} then follows by noting the
  integral equality
  $\int_0^t g_a(s) ds = |t - a| - |a|$.  The restriction to $a \ge 0$
  in the lemma follows
  since $\ltwo{v} \ge 0$ for all $v \in \R^\dim$.
\end{proof}

By Lemma~\ref{lemma:linear-to-absolute-value}, for any $1$-Lipschitz $h$, the
associated difference function has expectation bounded as
\begin{equation*}
  \E[\difffunc_\smoothparam(\perturbrv_1)]
  = \E[h(\perturbrv_1 + \smoothparam \perturbrv_2) - h(\perturbrv_1)]
  \le \sup_{a \in \R_+} \E\left[\left|
    \ltwo{\perturbrv_1 + \smoothparam \perturbrv_2}
    - a\right|
    - \left|\ltwo{\perturbrv_1} - a\right|\right].
\end{equation*}
For the distributions identified by
Assumption~\ref{assumption:smoothing-dist-general}, we can in fact show that
the preceding supremum is attained at $a = 0$.

\begin{lemma}
  \label{lemma:supremum-achieved-zero}
  Let $\perturbrv_1 \sim \smoothingdist_1$ and
  $\perturbrv_2 \sim \smoothingdist_2$ be independent, where 
  $\smoothingdist_1$ and $\smoothingdist_2$ satisfy
  Assumption~\ref{assumption:smoothing-dist-general}. For any
  $\smoothparam \ge 0$, the function
  \begin{equation*}
    a \mapsto \zeta(a) \defeq
    \E\left[|\ltwo{\perturbrv_1 + \smoothparam \perturbrv_2} - a|
     - |\ltwo{\perturbrv_1} - a|\right]
  \end{equation*}
  is non-increasing in $a \ge 0$.
\end{lemma}
\noindent
We return to prove this lemma at the end of the section.

With the intermediate results above, we can complete our proof of
Lemma~\ref{lemma:moment-bound}.
In view of Lemma~\ref{lemma:concentration-h}, we only need to bound
$\E[\difffunc_\smoothparam(\perturbrv_1)^k]$, where
$\difffunc_\smoothparam(\optvar) = \E[h(\optvar+\smoothparam \perturbrv_2)]
- h(\optvar)$. Recall that $\difffunc_\smoothparam(\optvar) \geq 0$ since
$h$ is convex. Moreover, since $h$ is $1$-Lipschitz,
\begin{equation*}
  \difffunc_\smoothparam(\optvar)
  \le \E\left[\big|h(\optvar+\smoothparam \perturbrv_2)
    - h(\optvar)\big|\right]
  \le \E\big[\|\smoothparam \perturbrv_2\|_2\big]
  \le u\E\big[\|\perturbrv_2\|_2^2\big]^{1/2}
  = \smoothparam \sqrt{\dim},
\end{equation*}
where the last equality follows from the choices of $\perturbrv_2$ in
Assumption~\ref{assumption:smoothing-dist-general}.  Therefore, we have the
crude but useful bound
\begin{equation}
  \label{eq:bound-eta-to-k}
  \E[\difffunc_\smoothparam(\perturbrv_1)^k]
  \le \smoothparam^{k-1} \dim^{\frac{k-1}{2}}
  \: \E[\difffunc_\smoothparam(\perturbrv_1)]
  = \smoothparam^{k-1} \dim^{\frac{k-1}{2}}\,
  \E[h(\perturbrv_1+\smoothparam \perturbrv_2)-h(\perturbrv_1)],
\end{equation}
where the last expectation is over both $\perturbrv_1$ and
$\perturbrv_2$. Since $\perturbrv_1$ and $\perturbrv_2$ both have rotationally
invariant distributions, Lemmas~\ref{lemma:linear-to-absolute-value}
and~\ref{lemma:supremum-achieved-zero} imply that the
expectation in expression~\eqref{eq:bound-eta-to-k}
is bounded by
\begin{equation*}
  \E[h(\perturbrv_1+\smoothparam \perturbrv_2)-h(\perturbrv_1)]
  \le \E\left[\ltwo{\perturbrv_1 + \smoothparam \perturbrv_2}
    - \ltwo{\perturbrv_1}\right].
\end{equation*}
Lemma~\ref{lemma:moment-bound} then follows by bounding the norm
difference in the preceding display for each choice of the smoothing
distributions in Assumption~\ref{assumption:smoothing-dist-general}.
We claim that
\begin{equation}
  \label{eqn:bound-expectation-diff}
  \E\left[\ltwo{\perturbrv_1 + \smoothparam \perturbrv_2}
    - \ltwo{\perturbrv_1}\right]
  \le \frac{1}{\sqrt{2}} \smoothparam^2 \sqrt{d}.
\end{equation}
To see this inequality, we consider the possible distributions for the pair
$\perturbrv_1, \perturbrv_2$ under
Assumption~\ref{assumption:smoothing-dist-general}.

\begin{enumerate}
\item  Let $T_\dim$ have $\chi^2$-distribution with
  $\dim$ degrees of freedom. Then
  for $\perturbrv_1, \perturbrv_2$ independent and
  $\normal(0, I_{d \times d})$-distributed,
  we have the distributional identities
  $\ltwo{\perturbrv_1 + \smoothparam \perturbrv_2} \eqd
  \sqrt{1+\smoothparam^2} \sqrt{T_\dim}$ and $\ltwo{\perturbrv_1}
  \eqd \sqrt{T_\dim}$. Using
  the inequalities $\sqrt{1+\smoothparam^2} \le 1 + \half \smoothparam^2$
  and $\E[\sqrt{T_\dim}] \le \E[T_\dim]^\half = \sqrt{\dim}$, we
  obtain
  \begin{equation*}
    \E\left[\ltwo{\perturbrv_1 + \smoothparam \perturbrv_2}
      - \ltwo{\perturbrv_1}\right]
    = \left(\sqrt{1+\smoothparam^2}-1\right) \E[\sqrt{T_\dim}]
    \le \half \smoothparam^2 \sqrt{\dim}.
  \end{equation*}
\item By Assumption~\ref{assumption:smoothing-dist-general},
  if $\perturbrv_1$ is uniform on $\sqrt{\dim+2}\,\B^d$ then
  $\perturbrv_2$ has either $\uniform(\sqrt{\dim+2}\,\B^d)$
  or $\uniform(\sqrt{\dim}\,\S^{\dim-1})$ distribution. Using the
  inequality $\sqrt{a+b} - \sqrt{a} \le b/(2\sqrt{a})$,
  valid for $a \ge 0$ and $b \ge -a$, we may write
  \begin{align*}
    \ltwo{\perturbrv_1 + \smoothparam \perturbrv_2} - \ltwo{\perturbrv_1}
    & = \sqrt{\ltwo{\perturbrv_1}^2 + 2\smoothparam \< \perturbrv_1, 
      \perturbrv_2 \> + \smoothparam^2 \ltwo{\perturbrv_2}^2}
    - \sqrt{\ltwo{\perturbrv_1}^2} \\
    & \le \frac{2\smoothparam \< \perturbrv_1, \perturbrv_2 \>
      + \smoothparam^2 \ltwo{\perturbrv_2}^2}{2\ltwo{\perturbrv_1}}
    = \smoothparam \< \frac{\perturbrv_1}{\ltwo{\perturbrv_1}}, \,
    \perturbrv_2 \> + \half \smoothparam^2
    \frac{\ltwo{\perturbrv_2}^2}{\ltwo{\perturbrv_1}}.
  \end{align*}
  Since $\perturbrv_1$ and $\perturbrv_2$ are independent and
  $\E[\perturbrv_2] = 0$, the expectation of the first term on the
  right hand side above vanishes. For the second term,
  the independence of $\perturbrv_1$ and $\perturbrv_2$
  and moment calculation~\eqref{eqn:moment-ball} imply
  \begin{equation*}
    \E\left[\ltwo{\perturbrv_1 + \smoothparam \perturbrv_2}
      - \ltwo{\perturbrv_1}\right]
    \le \half \smoothparam^2
    \,\E\left[\frac{1}{\ltwo{\perturbrv_1}}\right]
    \,\E\left[\ltwo{\perturbrv_2}^2\right]
    = \half \smoothparam^2\cdot\frac{1}{\sqrt{\dim+2}}\,
    \frac{\dim}{(\dim-1)}\cdot\dim
    \le \frac{1}{\sqrt{2}} \smoothparam^2 \sqrt{\dim},
  \end{equation*}
  where the last inequality holds for $d \ge 2$.
\end{enumerate}
\noindent
We thus obtain the claim~\eqref{eqn:bound-expectation-diff}, and
applying inequality~\eqref{eqn:bound-expectation-diff} to our earlier
computation~\eqref{eq:bound-eta-to-k} yields
\begin{equation*}
  \E[\difffunc_\smoothparam(\perturbrv_1)^k]
  \le \frac{1}{\sqrt{2}} \smoothparam^{k+1} \dim^{\frac{k}{2}}.
\end{equation*}
Plugging in this bound on $\difffunc_\smoothparam$ to
Lemma~\ref{lemma:concentration-h}, we obtain the result
\begin{align*}
  \E\left[\left|h(\perturbrv_1 + \smoothparam \perturbrv_2) -
    h(\perturbrv_1)\right|^k\right]
  &\le 2^{k-\frac{3}{2}} \smoothparam^{k+1} \dim^{\frac{k}{2}}
  + c^{-\frac{k}{2}} 2^{k-1} k^{\frac{k}{2}}
  \smoothparam^k \log^{\frac{k}{2}} (\dim+2k)
  + \sqrt{2} \smoothparam^k \\
  &\le c_k \smoothparam^k \left[
    \smoothparam \dim^{\frac{k}{2}}
    + 1 + \log^{\frac{k}{2}}(\dim+2k) \right],
\end{align*}
where $c_k$ is a numerical constant that only depends on $k$.
This is the desired statement of Lemma~\ref{lemma:moment-bound}.

We now return to prove the remaining intermediate lemma.
\paragraph{Proof of Lemma~\ref{lemma:supremum-achieved-zero}}
Since the quantity $\ltwo{\perturbrv_1 + \smoothparam \perturbrv_2}$ has a
density with respect to Lebesgue measure, standard results on differentiating
through an expectation~\cite[e.g.,][]{Bertsekas73} imply
\begin{equation*}
  \frac{d}{da} \E\left[\left|\ltwo{\perturbrv_1 + \smoothparam \perturbrv_2}
    - a \right| \right]
  = \E[\sign(a - \ltwo{\perturbrv_1 + \smoothparam \perturbrv_2})]
  = \P(\ltwo{\perturbrv_1 + \smoothparam \perturbrv_2} \le a)
  - \P(\ltwo{\perturbrv_1 + \smoothparam \perturbrv_2} > a),
\end{equation*}
where we used that the subdifferential of $a \mapsto |v - a|$ is
$\sign(a - v)$. As a consequence, we find that
\begin{align}
  \frac{d}{da} \zeta(a)
  & = \P(\ltwo{\perturbrv_1 + \smoothparam \perturbrv_2} \le a)
  - \P(\ltwo{\perturbrv_1 + \smoothparam \perturbrv_2} > a)
  - \P(\ltwo{\perturbrv_1} \le a) + \P(\ltwo{\perturbrv_1} > a)
  \nonumber \\
  & = 2\left[ \P\left(\ltwo{\perturbrv_1 + \smoothparam \perturbrv_2}
    \le a\right) 
    - \P\left(\ltwo{\perturbrv_1} \le a\right) \right].
  \label{eqn:derivative-zeta}
\end{align}
If we can show the quantity~\eqref{eqn:derivative-zeta} is non-positive
for all $a$, we obtain our desired result. It thus remains
to prove that $\ltwo{\perturbrv_1 + \smoothparam \perturbrv_2}$
stochastically dominates $\ltwo{\perturbrv_1}$ for each choice of
$\smoothingdist_1,\smoothingdist_2$ satisfying
Assumption~\ref{assumption:smoothing-dist-general}. We enumerate each
of the cases below.

\begin{enumerate}
\item 
  Let $T_d$ have $\chi^2$-distribution with $d$ degrees of freedom and
  $\perturbrv_1, \perturbrv_2 \sim \normal(0,I_{\dim \times \dim})$.  Then by
  definition we have $\ltwo{\perturbrv_1 + \smoothparam \perturbrv_2} \eqd
  \sqrt{1+\smoothparam^2} \sqrt{T_\dim}$ and $\ltwo{\perturbrv_1} \eqd
  \sqrt{T_\dim}$, and
  \begin{equation*}
    \P\left(\ltwo{\perturbrv_1 + \smoothparam \perturbrv_2}
    \le a\right)
    = \P\left(\sqrt{T_d} \le \frac{a}{\sqrt{1+\smoothparam^2}}\right)
    \le \P\left(\sqrt{T_d} \le a\right)
    = \P\left(\ltwo{\perturbrv_1} \le a\right)
  \end{equation*}
  as desired.
\item Now suppose $\perturbrv_1, \perturbrv_2$ are independent
  and distributed as $\uniform(r\,\B^\dim)$;
  our desired result will follow by setting $r = \sqrt{\dim+2}$.
  Let $p_0(t)$ and $p_\smoothparam(t)$ denote the densities of
  $\ltwo{\perturbrv_1}$ and $\ltwo{\perturbrv_1 + \smoothparam
    \perturbrv_2}$, respectively, with respect to Lebesgue measure on
  $\R$. We now compute them explicitly. For $p_0$, for $0 \le t \le r$
  we have
  \begin{equation*}
    p_0(t) = \frac{d}{dt} \P(\ltwo{\perturbrv_1} \le t)
    = \frac{d}{dt} \left(\frac{t}{r}\right)^\dim
    = \frac{\dim\,t^{\dim-1}}{r^\dim},
  \end{equation*}
  and $p_0(t) = 0$ otherwise. For $p_\smoothparam$, let $\lambda$ denote
  the Lebesgue measure in $\R^\dim$ and $\sigma$ denote the
  $(\dim-1)$-dimensional surface area in $\R^\dim$.
  The random variables $\perturbrv_1$ and $\smoothparam \perturbrv_2$
  have densities, respectively,
  \begin{equation*}
    q_1(x) = \frac{1}{\lambda(r\,\B^\dim)}
    = \frac{1}{r^\dim \lambda(\B^\dim)}
    \quad \text{ for } x \in r \B^\dim
  \end{equation*}
  and
  \begin{equation*}
    q_\smoothparam(x) = \frac{1}{\lambda(\smoothparam r \,\B^\dim)}
    = \frac{1}{\smoothparam^\dim r^\dim \lambda(\B^\dim)}
    \quad \text{ for } x \in \smoothparam r \B^\dim,
  \end{equation*}
  and $q_1(x) = q_\smoothparam(x) = 0$ otherwise. Then the density of
  $\perturbrv_1 + \smoothparam \perturbrv_2$ is given by the convolution
  \begin{align*}
    \tilde q(z)
    = \int_{\R^d} q_1(x) q_\smoothparam(z-x) \: \lambda(dx)
    = \int_{E(z)} \frac{1}{r^\dim \lambda(\B^\dim)} \cdot
    \frac{1}{\smoothparam^\dim r^\dim \lambda(\B^\dim)}
    \: \lambda(dx)
    = \frac{\lambda(E(z))}{\smoothparam^\dim\,r^{2\dim}
      \lambda(\B^\dim)^2}.
  \end{align*}
  Here $E(z) \defeq \B^\dim(0,r) \cap \B^\dim(z, \smoothparam r)$
  is the domain of integration, in which the densities $q_1(x)$ and
  $q_\smoothparam(z-x)$ are nonzero. The volume 
  $\lambda(E(z))$---and hence also $\tilde q(z)$---depend on $z$ only
  via its norm $\ltwo{z}$. Therefore, the density $p_\smoothparam(t)$
  of $\ltwo{\perturbrv_1 + \smoothparam \perturbrv_2}$ can be expressed
  as
  \begin{align*}
    p_\smoothparam(t)
    = \tilde q(t\basis_1)\,\sigma(t\S^{\dim-1})
    = \frac{\lambda(E(t\basis_1))\,t^{\dim-1}\,\sigma(\S^{\dim-1})}{
      \smoothparam^\dim\,r^{2\dim}\,\lambda(\B^\dim)^2}
    = \dim \, \frac{\lambda(E(t\basis_1))\,t^{\dim-1}}
    {\smoothparam^\dim\,r^{2\dim}\,\lambda(\B^\dim)},
  \end{align*}
  where the last equality above follows from the relation
  $\sigma(\S^{\dim-1}) = \dim \lambda(\B^\dim)$.
  Since $E(t\basis_1) \subseteq \B^\dim(t\basis_1,\smoothparam r)$
  by definition,
  \begin{equation*}
    \lambda(E(t\basis_1))
    \le \lambda\left(\B^\dim(t\basis_1,\smoothparam r)\right)
    = \smoothparam^\dim r^\dim \,\lambda(\B^\dim),
  \end{equation*}
  so for all $0 \le t \le (1+\smoothparam)r$ we have
  \begin{equation*}
    p_\smoothparam(t)
    = \dim \, \frac{\lambda(E(t\basis_1))\,t^{\dim-1}}
    {\smoothparam^\dim\,r^{2\dim}\,\lambda(\B^\dim)}
    \le \frac{\dim\,t^{\dim-1}}{r^\dim},
  \end{equation*}
  and clearly $p_\smoothparam(t) = 0$ for $t > (1+\smoothparam)r$.
  In particular, $p_\smoothparam(t) \le p_1(t)$ for $0 \le t \le r$,
  which gives us our desired stochastic dominance
  inequality~\eqref{eqn:derivative-zeta}: for $a \in [0, r]$,
  \begin{equation*}
    \P(\ltwo{\perturbrv_1 + \smoothparam \perturbrv_2} \le a)
    = \int_0^a p_\smoothparam(t)\,dt
    \le \int_0^a p_0(t)\,dt
    = \P(\ltwo{\perturbrv_1} \le a),
  \end{equation*}
  and for $a > r$ we have $\P(\ltwo{\perturbrv_1 + \smoothparam
    \perturbrv_2} \le a) \le 1 = \P(\ltwo{\perturbrv_1} \le a)$.
\item Finally, consider the case when $\perturbrv_1 \sim
  \uniform(\sqrt{\dim+2}\,\B^d)$ and $\perturbrv_2 \sim
  \uniform(\sqrt{\dim}\,\S^{\dim-1})$. As in the previous case,
  we will show that $p_0(t) \le p_\smoothparam(t)$ for $0 \le t \le
  \sqrt{\dim+2}$, where $p_0(t)$ and $p_\smoothparam(t)$ are the
  densities of $\ltwo{\perturbrv_1}$ and $\ltwo{\perturbrv_1 + \smoothparam
    \perturbrv_2}$, respectively. We know that the density of
  $\ltwo{\perturbrv_1}$ is
  \begin{equation*}
    p_0(t) = \frac{\dim\,t^{\dim-1}}{(\dim+2)^{\frac{\dim}{2}}}
    \quad \text{ for } 0 \le t \le \sqrt{\dim+2},
  \end{equation*}
  and $p_0(t) = 0$ otherwise. To compute $p_\smoothparam$, we first
  determine the density $\tilde q(z)$ of the random variable $\perturbrv_1 +
  \smoothparam \perturbrv_2$ with respect to the Lebesgue measure $\lambda$
  on $\R^\dim$. The usual convolution formula does not directly apply as
  $\perturbrv_1$ and $\perturbrv_2$ have densities with respect to different
  base measures ($\lambda$ and $\sigma$, respectively). However, as
  $\perturbrv_1$ and $\perturbrv_2$ are both uniform, we can argue as
  follows. Integrating over the surface $\smoothparam \sqrt{\dim} \S^{\dim -
    1}$ (essentially performing a convolution), each point $\smoothparam y
  \in \smoothparam \sqrt{\dim}\, \S^{\dim-1}$ contributes the amount
  \begin{equation*}
    \frac{1}{\sigma(\smoothparam \sqrt{\dim}\,\S^{\dim-1})} \cdot
    \frac{1}{\lambda(\sqrt{\dim+2}\,\B^\dim)}
    = \frac{1}{\smoothparam^{\dim-1}\,\dim^{\frac{\dim-1}{2}}\,
      (\dim+2)^{\frac{\dim}{2}}\,\sigma(\S^{\dim-1})\,\lambda(\B^\dim)}
  \end{equation*}
  to the density $\tilde q(z)$, provided $\ltwo{z-\smoothparam y} \le
  \sqrt{\dim+2}$. For fixed $z \in (\sqrt{\dim + 2} + \smoothparam
  \sqrt{\dim}) \B^\dim$, the set of such contributing points $uy$ can be
  written as $E(z) = \B^\dim(z,\sqrt{\dim+2}) \cap
  \S^{\dim-1}(0,u\sqrt{\dim})$. Therefore, the density of $\perturbrv_1 +
  \smoothparam \perturbrv_2$ is given by
  \begin{equation*}
    \tilde q(z) = \frac{\sigma(E(z))}
           {\smoothparam^{\dim-1}\,\dim^{\frac{\dim-1}{2}}\,
             (\dim+2)^{\frac{\dim}{2}}\,\sigma(\S^{\dim-1})\,\lambda(\B^\dim)}.
  \end{equation*}
  Since $\tilde q(z)$ only depends on $z$ via its norm $\ltwo{z}$,
  the formula above also gives us the density $p_\smoothparam(t)$
  of $\ltwo{\perturbrv_1 + \smoothparam \perturbrv_2}$:
  \begin{equation*}
    p_\smoothparam(t) = \tilde q(t\basis_1) \, \sigma(t\S^{\dim-1})
    = \frac{\sigma(E(z))\,t^{\dim-1}}
    {\smoothparam^{\dim-1}\,\dim^{\frac{\dim-1}{2}}\,
      (\dim+2)^{\frac{\dim}{2}}\,\lambda(\B^\dim)}.
  \end{equation*}
  Noting that $E(z) \subseteq \S^{\dim-1}(0,\smoothparam \sqrt{\dim})$
  gives us
  \begin{equation*}
    p_\smoothparam(t)
    \le \frac{\sigma(\smoothparam \sqrt{\dim}\,\S^{\dim-1})\,t^{\dim-1}}
        {\smoothparam^{\dim-1}\,\dim^{\frac{\dim-1}{2}}\,
          (\dim+2)^{\frac{\dim}{2}}\,\lambda(\B^\dim)}
        = \frac{\dim\,t^{\dim-1}}{(\dim+2)^{\frac{\dim}{2}}}.
  \end{equation*}
  In particular, we have $p_\smoothparam(t) \le p_0(t)$ for $0 \le t
  \le \sqrt{\dim+2}$, which, as we saw in the previous case, gives us
  the desired stochastic dominance
  inequality~\eqref{eqn:derivative-zeta}.
\end{enumerate}


\section{Technical proofs associated with lower bounds}
\label{sec:proof-lower-bound}

In this section, we prove the technical results necessary
for the proofs of Propositions~\ref{proposition:ltwo-lower-bound}
and~\ref{proposition:lone-lower-bound}.

\subsection{Proof of Lemma~\ref{lemma:onevec-optimization}}
\label{sec:proof-onevec-optimization}

First, note that the optimal vector $\optvar^A = -d^{-1/q}\onevec$ with
optimal value $-d^{1 - 1/q}$, and $\optvar^B = -(d-i)^{-1/q} \onevec_{i+1:d}$,
where $\onevec_{i+1:d}$ denotes the vector with 0 entries in its first $i$
coordinates and $1$ elsewhere. As a consequence, we have $\<\optvar^B,
\onevec\> = -(d - i)^{1-1/q}$. Now we use the fact that by convexity of the
function $x \mapsto -x^{1 - 1/q}$ for $q \in [1, \infty]$,
\begin{equation*}
  - d^{1 - 1/q} \le - (d - i)^{1 - 1/q}
  - \frac{1 - 1/q}{d^{1/q}} i,
\end{equation*}
since the derivative of $x \mapsto -x^{1-1/q}$ at $x = d$ is given by $-(1 -
1/q) / d^{1/q}$ and the quantity $-x^{1-1/q}$ is non-increasing in $x$ for $q
\in [1, \infty]$.

\subsection{Proof of Lemma~\ref{lemma:coordinate-tv-bound}}
\label{sec:proof-coordinate-tv-bound}

For notational convenience, let the distribution $\statprob_{\packval,+j}$ be
identical to the distribution $\statprob_\packval$ but with the $j$th
coordinate $\packval_j$ forced to be $+1$ and similarly for
$\statprob_{\packval,-j}$. Using Pinsker's inequality and the joint convexity
of the KL-divergence, we have
\begin{equation*}
\begin{split}
  \tvnorm{\statprob_{+j} - \statprob_{-j}}^2
  &\le \frac{1}{4} \left[\dkl{\statprob_{+j}}{\statprob_{-j}}
  + \dkl{\statprob_{-j}}{\statprob_{+j}}\right] \\
  &\le \frac{1}{2^{d+2}} \sum_{\packval \in \packset}
  \left[\dkl{\statprob_{\packval,+j}}{\statprob_{\packval,-j}}
  + \dkl{\statprob_{\packval,-j}}{\statprob_{\packval,+j}}\right].
\end{split}
\end{equation*}
By the chain-rule for KL-divergences~\cite{CoverTh06}, if we define
$\statprob_\packval^\iter(\cdot \mid \obs[1:\iter-1])$ to be the distribution
of the $\iter^{\text{th}}$ observation $\obs[\iter]$ conditional on
$\packval$ and $\obs[1:\iter-1]$, then we have
\begin{equation*}
  \dkl{\statprob_{\packval,+j}}{\statprob_{\packval,-j}}
  = \sum_{\iter = 1}^\totaliter
  \int_{\obsdomain^{\iter-1}} \dkl{\statprob^\iter_{\packval,+j}(\cdot
    \mid \obs[1:\iter-1] = \obsval)}{\statprob^\iter_{\packval,-j}(\cdot
    \mid \obs[1:\iter-1] = \obsval)} d \statprob_{\packval,+j}(\obsval).
\end{equation*}

We show how to bound the preceding sequence of KL-divergences for the
observational scheme based on function-evaluations we allow.  Let $\pairoptvar
= [\optvar ~ \altoptvar] \in \R^{d \times 2}$ denote the pair of query points,
so we have by construction that the observation $\obs = \pairoptvar^\top
\statrv$ where $\statrv \mid \packrv = \packval \sim \normal(\delta \packval,
\stddev^2 I_{d \times d})$.  In particular, given $\packval$ and the pair
$\pairoptvar$, the vector $\obs \in \R^d$ is normally distributed with mean
$\delta \pairoptvar^\top \packval$ and covariance $\stddev^2 \pairoptvar^\top
\pairoptvar = \Sigma$, where the covariance $\Sigma$ is defined in
equation~\eqref{eqn:compute-covariance}.
The KL divergence between normal
distributions is $\dkl{\normal(\mu_1, \Sigma)}{\normal(\mu_2, \Sigma)} = \half
(\mu_1 - \mu_2)^\top \Sigma^{-1} (\mu_1 - \mu_2)$. Note that if
$\packval$ and
$\altpackval$ differ in only coordinate $j$, then $\<\packval - \altpackval,
\optvar\> = (\packval_j - \altpackval_j)\optvar_j$. We thus obtain
\begin{equation*}
  \dkl{\statprob^\iter_{\packval,+j}(\cdot \mid \obsval[1:\iter-1])}{
    \statprob^\iter_{\packval,-j}(\cdot \mid \obsval[1:\iter-1])}
  \le 2 \delta^2 \E\left[
    \left[\begin{matrix} \optvar[\iter]_j \\ \altoptvar[\iter]_j
      \end{matrix}\right]^\top
    (\Sigma^\iter)^{-1}
    \left[\begin{matrix} \optvar[\iter]_j \\ \altoptvar[\iter]_j
      \end{matrix}\right]
    \mid \obsval[1:\iter-1]\right]
\end{equation*}
where the expectation is taken with respect to any additional randomness in
the construction of the pair $(\optvar[\iter], \altoptvar[\iter])$ (as,
aside from this randomness, they are measureable $\obs[1:\totaliter-1]$). We
obtain an identical bound for $\dkls{\statprob^\iter_{\packval,-j}(\cdot
  \mid \obsval[1:\iter-1])}{\statprob^\iter_{\packval,+j}(\cdot \mid
  \obsval[1:\iter-1])}$. Combining the sequence of inequalities from the
preceding paragraph, we thus obtain
\begin{align*}
  \tvnorm{\statprob_{+j} - \statprob_{-j}}^2
  & \le \frac{\delta^2}{2^{d+1}} \sum_{\iter = 1}^\totaliter
  \sum_{\packval \in \packset}
  \int_{\obsdomain^{\iter - 1}} 
  \E\left[
    \left[\begin{matrix} \optvar[\iter]_j \\ \altoptvar[\iter]_j
      \end{matrix}\right]^\top
    (\Sigma^\iter)^{-1}
    \left[\begin{matrix} \optvar[\iter]_j \\ \altoptvar[\iter]_j
      \end{matrix}\right]
    \mid \obsval[1:\iter-1]\right]
  (d\statprob_{\packval,+j}(\obsval[1:\iter-1])
  + d\statprob_{\packval,-j}(\obsval[1:\iter-1])) \nonumber \\
  & = \frac{\delta^2}{2} \sum_{\iter=1}^\totaliter
  \int_{\obsdomain^{\iter-1}}
  \E\left[
    \left[\begin{matrix} \optvar[\iter]_j \\ \altoptvar[\iter]_j
      \end{matrix}\right]^\top
    (\Sigma^\iter)^{-1}
    \left[\begin{matrix} \optvar[\iter]_j \\ \altoptvar[\iter]_j
      \end{matrix}\right]
    \mid \obsval[1:\iter-1]\right]
  \left(d\statprob_{+j}(\obsval[1:\iter-1])
  + d\statprob_{-j}(\obsval[1:\iter-1])\right),
\end{align*}
where for the equality we used the definitions of
the distributions $\statprob_{\packval,\pm j}$ and $\statprob_{\pm j}$.
Integrating over the observations $\obsval$ proves
the claimed inequality~\eqref{eqn:super-j-tv-bound}.

{\small
\bibliographystyle{abbrvnat}
\bibliography{bib}
}

\end{document}